\documentclass{mcom-l}

\usepackage[english]{babel}
\usepackage[utf8]{inputenc}

\usepackage{amsthm,amsmath,amsfonts,amssymb}
\usepackage{hyperref}
\usepackage{graphicx}
\usepackage{enumerate}
\usepackage{nicefrac}
\usepackage{paralist}
\usepackage{pstool}
\usepackage{booktabs}
\usepackage{mathtools}
\usepackage{url}

\usepackage{xcolor} 
\usepackage{graphics}
\usepackage{mathrsfs}
\usepackage{mathtools}            
\usepackage{subfigure}
\usepackage{nicefrac}
\usepackage{paralist}
\usepackage[inline]{enumitem}

\DeclareMathOperator{\supp}{supp}

\DeclareMathOperator{\diag}{diag}
\newcommand{\proper}{\mathsf}

\newcommand{\pN}{\proper{N}}
\newcommand{\mv}[1]{{\boldsymbol{\mathrm{#1}}}}

\newcommand{\md}{\ensuremath{\,\mathrm{d}}}
\DeclarePairedDelimiter\ceil{\lceil}{\rceil}

\renewcommand{\H}{a}

\graphicspath{{../tex/figs/}}

\allowdisplaybreaks

\newtheorem{Theorem}{Theorem}[section]
\newtheorem{Lemma}[Theorem]{Lemma}
\newtheorem{Proposition}[Theorem]{Proposition}
\newtheorem{Corollary}[Theorem]{Corollary}

\theoremstyle{definition}

\newtheorem{example}[Theorem]{Example}
\newtheorem{Assumption}[Theorem]{Assumption}

\theoremstyle{remark}
\newtheorem{Remark}[Theorem]{Remark}

\begin{document}

\title[Numerics for fractional differential equations on metric graphs]{Regularity and numerical approximation of fractional elliptic differential equations on compact metric graphs}


\author{David Bolin}
\address[David Bolin]{Statistics Program, Computer, Electrical and Mathematical Sciences and Engineering Division,
	King Abdullah University of Science and Technology,
	Thuwal 23955-6900,
	Saudi Arabia}	
\email{david.bolin@kaust.edu.sa}

\author{Mih\'aly Kov\'acs}
\address[Mih\'aly Kov\'acs]{Department of Mathematical Sciences,Chalmers University of Technology and University of Gothenburg, SE-41296 Gothenburg, Sweden\newline
	Department of Differential Equations, Budapest University of Technology and Economics, M\H{u}egyetem rkp. 3., H-1111 Budapest, Hungary\newline
	Faculty of Information Technology and Bionics, P\'azm\'any P\'eter Catholic University, Pr\'ater utca 50/a., H-1083 Budapest, Hungary}
\email{mihaly@chalmers.se}
\thanks{}

\author{Vivek Kumar}
\address[Vivek Kumar]{Theoretical Statistics and Mathematics Unit, Indian Statistical Institute, Bangalore Centre, 8th Mile Mysore Road, Bangalore, 560059, Karnataka, India}
\email{vivekkumar\_ra@isibang.ac.in}
\thanks{}

\author{Alexandre B.\ Simas}
\address[Alexandre B.\ Simas]{Statistics Program, Computer, Electrical and Mathematical Sciences and Engineering Division \\
	King Abdullah University of Science and Technology\\
	Thuwal 23955-6900\\
	Saudi Arabia}	
\email{alexandre.simas@kaust.edu.sa}
\thanks{}


\date{}

\dedicatory{}

\begin{abstract}
	The fractional differential equation $L^\beta u = f$ posed on a compact metric graph is considered, where $\beta>0$ and $L = \kappa^2 - \nabla(\H\nabla)$ is a second-order elliptic operator equipped with certain vertex conditions and sufficiently smooth and positive coefficients $\kappa,\H$. We demonstrate the existence of a unique solution for a general class of vertex conditions and derive the regularity of the solution in the specific case of Kirchhoff vertex conditions. These results are extended to the stochastic setting when $f$ is replaced by Gaussian white noise. For the deterministic and stochastic settings under generalized Kirchhoff vertex conditions, we propose a numerical solution based on a finite element approximation combined with a rational approximation of the fractional power $L^{-\beta}$. For the resulting approximation, the strong error is analyzed in the deterministic case, and the strong mean squared error as well as the $L_2(\Gamma\times \Gamma)$-error of the covariance function of the solution are analyzed in the stochastic setting. Explicit rates of convergences are derived for all cases. Numerical experiments for ${L = \kappa^2 - \Delta, \kappa>0}$ are performed to illustrate the  results.
\end{abstract}

\maketitle

\section{Introduction and preliminaries}

\subsection{Introduction}
Recently, the study of differential operators on metric graphs has gained considerable attention 
\cite{Berkolaiko2013}. 
Metric graphs equipped with differential operators, also known as quantum graphs, are applied in a broad range of applications, such as the formulation of free-electron models for organic molecules \cite{pauling1936diamagnetic}, superconductivity in granular and artificial materials \cite{alexander1983superconductivity}, acoustic and electromagnetic wave guide networks  \cite{flesia1987strong}, Anderson transition in a disordered wire \cite{anderes2020isotropic,shapiro1982renormalization}, description of quantum chaos \cite{Berkolaiko2013} and statistics \cite{BSW2022}. 

A compact metric graph $\Gamma$ consists of a finite set of vertices $\mathcal{V}=\{v_i\}$ and a finite set of edges 
$\mathcal{E}=\{e_j\}$ connecting the vertices. Each edge can be considered as a parametric curve with a given length 
$l_e \in (0,\infty)$, which connects a pair of vertices. 
A location $s =(e,t) \in \Gamma$ is a position on an edge, $e$, where $t\in[0,l_e]$. As a metric on $\Gamma$, 
we use the shortest path distance, which for any two points in $\Gamma$ is defined as the length of the shortest path in $\Gamma$ connecting the two.
The degree of each vertex (i.e., the number of edges connected to it) is finite, and we assume that the graph is connected, so that there is a path between any two locations in the graph. 
On such a graph, we for $\beta>0$ consider the following fractional elliptic differential equation:
\begin{align}\label{Amaineqn_deterministic}
	L^\beta u=f \quad \mbox{on $\Gamma$,}
\end{align}
where $f\in L_{2}(\Gamma)$. Here, $L_{2}(\Gamma)$ is the space of real-valued measurable functions, which are square-integrable on each $e\in \mathcal{E}$. That is, the restriction of $f$ to $e$, denoted by $f_e$, belongs to $L_2(e)$, the space of square-integrable functions on $e$. The space $L_2(\Gamma)$ is equipped with the norm $\|u\|_{L_{2}(\Gamma)}^{2}:=\sum_{e\in \mathcal{E}}\|u_e\|_{L_2(e)}^{2}$ and inner product $(u,v)=\sum_{e\in\mathcal{E}}\int_e u_e(x)v_e(x)dx$. Further, $L$ is a second-order elliptic operator that acts on a sufficiently differentiable function $u$ at a location $x$ on $e\in\mathcal{E}$, as follows:
\begin{equation}\label{eq:operatorL}
	(Lu)(x) = -\frac{\md}{\md x_e}\left(\H_e(x) \frac{\md}{\md x_e} u_e(x)\right) + \kappa_e^2(x)u_e(x),
\end{equation}
where $\md /\md x_e$ is the standard derivative operator on $e$, when we identify $e$ with the interval $[0,l_e]$ and $\kappa$ is a bounded function, $\kappa\in L_\infty(\Gamma)$, satisfying
\begin{equation}\label{eq:condcoeff2}
	\textrm{ess inf}_{x\in\Gamma} \kappa(x) \geq \kappa_0 > 0,
\end{equation}
for some  $\kappa_0>0$, and $\H:\Gamma\to\mathbb{R}$ is a positive Lipschitz function. As $\Gamma$ is compact, the minimum of $\H$ is achieved for some $x_0\in\Gamma$. We let $\H_0 = \H(x_0)$; thus,
\begin{equation}\label{eq:condcoeff1}
	\forall x\in\Gamma, \quad \H(x) \geq \H_0 > 0.
\end{equation}
The operator is also assumed to satisfy certain vertex conditions, at the vertices of $\Gamma$. Finally, the fractional power $L^\beta$ is understood in the spectral sense, as discussed in detail later together with the exact domain of the operator. One example of vertex conditions that we consider is the generalized Kirchhoff conditions:
\begin{equation}\label{eq:kirchhoff}
	\big\{\mbox{$u$ is continuous on $\Gamma$ and $\forall v\in\mathcal{V}: \H(v)\sum_{e\in\mathcal{E}_v} \partial_e u(v) = \alpha u(v)$} \big\},
\end{equation}
where $\partial_eu(v)$ is the directional derivative of $u$ on $e$ in the directional away from $v$, that is, if $e=[0,l_e]$, then $\partial_e u(0) = \nicefrac{\md u}{\md x_e}(0)$, and $\partial_eu(l_e) = -\nicefrac{\md u}{\md x_e}(l_e)$, $\alpha\in\mathbb{R}$ and $\mathcal{E}_v$ denotes the set of edges connected to the vertex $v$. For $\alpha=0$, the commonly used standard Kirchhoff conditions are obtained. 

We also consider the fractional stochastic differential equation:
\begin{align}\label{Amaineqn}
	L^\beta u=\mathcal{W}\quad\text{on } \Gamma,
\end{align}
where $\beta>\nicefrac14$ and $\mathcal{W}$ represents Gaussian white noise defined on a given probability space $(\Omega, \mathcal{F},\mathbb{P})$.
The importance of Eq.~\eqref{Amaineqn} is that, when considered on $\mathbb{R}^d$ instead of $\Gamma$ with $\kappa$ is constant and $\H=I$, it defines the popular Gaussian Mat\'{e}rn fields, which are often used for modeling in spatial statistics \cite{lindgren2022spde}. Gaussian random fields, including Mat\'{e}rn fields, have been extensively studied on Euclidean spaces, but the corresponding results metric graphs are much more limited \cite{anderes2020isotropic,borovitskiy2021matern}. One reason for this is that it is challenging to define valid isotropic covariance functions \cite{anderes2020isotropic} in the metric of the graph.
However, \cite{BSW2022} used the SPDE connection for Gaussian Mat\'ern fields to define Gaussian Whittle--Mat\'ern fields on compact metric graphs as solutions to \eqref{Amaineqn} when $\kappa$ is constant and $\H=I$. These models have recently been shown to be valuable for real-world applications in \cite{BSW2023_AOS}.
When Eq.~\eqref{Amaineqn} is considered on $\mathbb{R}^d$ with non-constant parameters $\kappa^2$ and $\H$, the resulting Gaussian fields are often referred to as generalized Whittle--Mat\'ern fields  \cite{lindgren2022spde, coxkirchner}. Such fields are important in statistics since they facilitate modeling data with non-stationary features such as a spatially varying marginal variance (see, e.g., \cite[see, e.g.]{BK2020rational}). We refer to the solution to \eqref{Amaineqn} as a generalized Whittle--Mat\'ern field on the metric graph, and we believe that these fields will be important for statistical applications requiring non-stationary Gaussian fields on metric graphs. In fact, as far as we know, \eqref{Amaineqn} is the first construction of Gaussian fields on general metric graphs that can model non-stationary features. 

\subsection{Summary of contributions and outline}
The paper is divided into three parts. In the first part (Sections~\ref{sec:operator} and \ref{sec:solutions}), we derive conditions for the existence and uniqueness of the solution to both \eqref{Amaineqn_deterministic} and \eqref{Amaineqn} given a general class of vertex conditions that includes the generalized Kirchhoff conditions as a particular case. This derivation extends the work of \cite{BSW2022} in which it was proved that \eqref{Amaineqn} has a unique solution for Kirchhoff vertex conditions, a constant $\kappa>0$, and $\H = I$. In this section, we also derive a new Weyl's law and a trace theorem for metric graphs.

In the second part (Section~\ref{sec:regularity} and \ref{sec:proof}), we focus on the case of standard Kirchhoff vertex conditions ($\alpha=0$) and characterize the regularity of the solutions in deterministic and stochastic cases. To do so, we first derive a new characterization of fractional Sobolev spaces on compact metric graphs in Theorem~\ref{thm:characterization}, which is one of our main results, as well as a new integration by parts formula for metric graphs. 
We note that these regularity results would be essential in the derivation of error estimates in Sobolev norms, similar to those in \cite{coxkirchner} for Gaussian random fields on Euclidean domains. However, we do not pursue this direction in the present paper.

In the final part, we propose a numerical method for approximating the solutions of \eqref{Amaineqn_deterministic} and \eqref{Amaineqn} for generalized Kirchhoff vertex conditions.
The numerical approximations, presented in Section~\ref{sec:numerical}, use a finite element method (FEM) combined with a quadrature approximation of the fractional power from \cite{bonito2015numerical}. The FEM discretization is inspired by \cite{Arioli2017FEM}, who considered the nonfractional deterministic elliptic equation \eqref{Amaineqn_deterministic} with $\H = I$ and $\beta=1$ and standard Kirchhoff vertex conditions, and we thus extend that work to more general operators, more general vertex conditions, fractional powers, and stochastic equations. 
For the proposed numerical method, we derive explicit rates of convergence of the strong mean squared error for the approximations of \eqref{Amaineqn_deterministic} and \eqref{Amaineqn}, and of the $L_2(\Gamma\times\Gamma)$-error of the approximation of the covariance function for the Gaussian random field defined by \eqref{Amaineqn}. 
Finally, the derived convergence rates are verified numerically in Section~\ref{sec:experiments}.

\subsection{Preliminaries and notation}\label{sec:preliminaries}
Throughout the paper, we let $\Gamma$ denote a general metric graph, which might have loops and multiple edges connecting the same two vertices. We further do not restrict the graph to have Euclidean edges as in \cite{anderes2020isotropic}.
We let $C(\Gamma) = \{f\in L_2(\Gamma): f$ is continuous$\}$ denote the space of continuous functions on $\Gamma$ and let 
$\|\phi\|_{C(\Gamma)} = \sup\{ \lvert\phi(x)\rvert : x\in\Gamma\}, \phi\in C(\Gamma)$, 
denote the supremum norm. For $0<\gamma<1$, we introduce the $\gamma$-H\"older semi-norm
$$
[\phi]_{C^{0,\gamma}(\Gamma)} = \sup\left\{\frac{\lvert\phi(x)-\phi(y)\rvert}{d(x,y)^{\gamma}}: x,y\in \Gamma, x\neq y \right\},
$$
and the $\gamma$-H\"older norm
$\|\phi\|_{C^{0,\gamma}(\Gamma)} = \|\phi\|_{C(\Gamma)} + [\phi]_{C^{0,\gamma}(\Gamma)}$. Then, $C^{0,\gamma}(\Gamma)$ denotes the space of $\gamma$-H\"older continuous functions (i.e., the set of functions $\phi\in C(\Gamma)$ with $\|\phi\|_{C^{0,\gamma}(\Gamma)}<\infty$). 
We let $H^{1}(\Gamma)$ denote the Sobolev space consisting of all continuous functions on $\Gamma$ that belong to $H^{1}(e)$ for each edge $e\in \mathcal{E}$, with norm
\begin{align*}
	\|u\|_{H^{1}(\Gamma)}^{2} :=& \sum_{e\in \mathcal{E}}	\|u\|_{H^{1}(e)}^{2} = \sum_{e\in \mathcal{E}} \int_{e}\left(u_e'(x)\right)^2dx+\int_{e}\lvert u\rvert^2dx.
\end{align*}  
Let $\nabla:H^1(\Gamma)\to L^2(\Gamma)$
	be the global derivative operator that acts locally on each edge $e\in\mathcal{E}$ as
	$\nabla u(x) = \md u(x)/\md x_e = u_e'(x) = Du_e(x)$, for $u\in H^1(\Gamma)$ and $x\in e$. 

For $r\in\mathbb{N}$, $H^r(e)$ is the standard higher-order Sobolev space on an interval. Further, for $r\in\mathbb{N}$, we introduce the decoupled Sobolev spaces ${\widetilde{H}}^{r}(\Gamma):=\bigoplus_{e\in \mathcal{E}}H^r(e)$, consisting of functions $u$ with norm
	$\|u\|_{{\widetilde{H}}^{r}(\Gamma)}^{2}:=\sum_{e\in \mathcal{E}}	\|u\|_{{H}^{r}(e)}^{2}<\infty$.
Thus, we enforce the smoothness conditions on the edges only and have no global continuity requirement.
Finally, let  ${\widetilde{H}^r_C(\Gamma) = \widetilde{H}^r(\Gamma)\cap C(\Gamma)}$, where $r\in\mathbb{N}$. We observe that $H^1(\Gamma) = \widetilde{H}^1_C(\Gamma)$ and, for $r\geq 1$, we have $\widetilde{H}^r_C(\Gamma)\subset H^1(\Gamma)$. 

Let us now introduce some general required notation. 
Let $(E,\|\cdot\|_E)$ and $(F,\|\cdot\|_F)$ be two separable Hilbert spaces with norms $\|\cdot\|_E$ and $\|\cdot\|_F$, respectively. If $E\subset F$ and there exists a constant $C>0$ such that for every $f\in E$, ${\|f\|_F\leq C\|f\|_E}$, we write $(E, \|\cdot\|_E) \hookrightarrow (F,\|\cdot\|_F)$, and say that $E$ is continuously embedded in $F$. If  ${(E, \|\cdot\|_E) \hookrightarrow (F,\|\cdot\|_F) \hookrightarrow (E,\|\cdot\|_E)}$, we write ${(E,\|\cdot\|_E) \cong (F,\|\cdot\|_F)}$. Further, $\mathcal{L}(E, F)$ is the Banach
space of bounded linear operators from $E$ to $F$, endowed with the operator norm  ${\|\cdot\|_{\mathcal{L}(E,F)} = \sup_{\|u\|_{E}=1} \|\cdot u\|_F}$. Similarly, $\mathcal{L}_2(E, F)$ is the Hilbert space of Hilbert--Schmidt operators, endowed with the Hilbert--Schmidt norm ${\|\cdot \|_{\mathcal{L}_2(E, F)}^{2}:= \sum_{i\in \mathbb{N}}\|\cdot e_i\|^{2}_{F}}$, where $\{e_i\}_{i\in \mathbb{N}}$ is a complete orthonormal set in $(E, \|\cdot\|_{E})$. To simplify the notation, $\mathcal{L}(E, E)$ is denoted by $\mathcal{L}(E)$ and $\mathcal{L}_2(E, E)$ is denoted by $\mathcal{L}_2(E)$.
Finally, for a parameter set $P$ and two mappings $A,B : P \rightarrow \mathbb{R}$, we write $A(p) \lesssim B(p)$ if there exists a constant $C>0$ such that $A(p) \leq B(p)$ for all $p\in P$. To indicate that $C$ depends on some parameter $q$ of the mappings, we write $A(p) \lesssim_q B(p)$. We define $C^{\infty}_{c}(e)$ as the set of infinitely differentiable functions with compact support on the interior of $e\in \mathcal{E}$. Finally, given a Hilbert space $H$ and its dual $H^*$,  $\langle\cdot,\cdot\rangle_{H^*\times H}$ is the duality pairing between $H^*$ and $H$. 

\section{The differential operator}\label{sec:operator}
This section introduces the differential operator $L$ for a general class of vertex conditions. We derive some basic properties of the operator that are used in the next section to prove the existence and uniqueness of the solutions to \eqref{Amaineqn_deterministic} and \eqref{Amaineqn}.

\subsection{Main properties}
Observe that, for any $u\in \bigoplus_{e\in\mathcal{E}} C^2(e)$, 
the operator $L$ acts on $u$ in a direction-free manner: if we replace the derivative $\frac{\md}{\md x_e}$ with $\frac{d}{d\overline{x}_{e}}$, where we consider the derivative in the reverse direction in the latter, then
$$-\frac{\md}{\md x_e}\left(\H(x) \frac{\md}{\md x_e} u(x)\right) + \kappa^2(x)u(x) = -\frac{d}{d\overline{x}_{e}}\left(\H(x) \frac{d}{d\overline{x}_{e}} u(x)\right) + \kappa^2(x)u(x).$$

  To define the vertex conditions of $L$, we introduce the degree $d_v$ of any vertex $v$ as $d_v=\lvert\mathcal{E}_v\rvert$, where $\mathcal{E}_v$ denotes the set of edges incident to vertex $v$. We let $u_i(v)$ denote the value at vertex $v$ that function $u$ attains along the edge $e_i$ incident to $v$, for $i=1,2,\ldots,d_v,$, that is, $u_i(v)$ is the restriction of $u(\cdot)$ to the edge $e_i$ evaluated at $v$. A commonly considered class of vertex 
conditions can be formulated as follows:
\begin{align}\label{BCs}
	A_vF(v)+\H(v)B_vF^{'}(v)=0~~~\mbox{for all}~ v \in \mathcal{V},
\end{align}
where $A_v, B_v$ are $d_v\times d_v$ matrices such that the matrix $(A_v,B_v)$, which consists of the columns of $A_v$ and $B_v$, has maximal rank $d_v$ and
$A_vB_{v}^{*}$ is self-adjoint. Further, $F(v) = (u_1(v),\ldots, u_{d_v}(v))^\top$ and $F^{'}(v):=(\partial_e u_{1}(v),\ldots, \partial_e u_{d_v}(v))^\top$.
The condition \eqref{BCs}, where $A_v=\mathbb{I}_{d_v}$ and $B_v=0$, provides Dirichlet boundary conditions, and $B_v=\mathbb{I}_{d_v}$ and $A_v=0$ results in Neumann boundary conditions. Moreover, if
\begin{equation}\label{eq:generalized_kirchoff}
	A_v = \begin{bmatrix}
		1 & -1     &        &  \\
		& \ddots & \ddots & \\
		&        & 1      & -1\\
		-\alpha  &       &  & 
	\end{bmatrix},
	\quad
	B_v=\begin{bmatrix}
		& & & & \\
		& & & & \\
		& & & & \\
		& & & & \\
		1 &1  &1  &\cdots &1
	\end{bmatrix},
\end{equation} 
we obtain the generalized Kirchhoff conditions \eqref{eq:kirchhoff}.

Let $P_v$ and $P_{1v}$ be orthogonal projections in $\mathbb{R}^{d_v}$ onto the 
kernels ${K_v=ker(B_v)}$ and $K_{1v}=ker(B^{*}_{v})$, respectively. 
We further let $Q_v = I-P_v$ and ${Q_{1v} = I - P_{1v}}$ be complementary orthogonal projectors on the 
ranges $R_v=Range(B^{*}_{v})$ and ${R_{1v}=Range(B_v)}$, respectively. Then, the map 
$Q_{1v}B_vQ_v:R_v\rightarrow R_{1v}$ is invertible (see \cite[Lemma~4]{kuchment2004quantum}), and we let $B_v^{(-1)}$ denote its inverse. Then, by
{\cite[Corollary 5]{kuchment2004quantum}}, the vertex conditions in \eqref{BCs} with the assumption that $(A_v, B_v )$ has a maximal rank $d_v$ and $A_vB^{*}_{v}$ is self-adjoint, 
is equivalent to the following pair of conditions:
\begin{equation}\label{eq:boundaryconditionsselfadjoint}
	P_vF(v)=0,\quad\hbox{and}\quad \Lambda_v Q_vF(v)+\H(v)Q_vF'(v)=0.
\end{equation}
where $\Lambda_v$ is the self-adjoint operator $B_v^{(-1)}A_v$. Under condition \eqref{eq:kirchhoff}, $\Lambda_v$ is the multiplication operator by the number $-\frac{\alpha}{d_v}$ \cite{kuchment2004quantum}.

\begin{example}\label{ex1}
	Let $v\in\mathcal{V}$ be a vertex of degree $3$ that satisfies 
	a generalized Kirchoff vertex condition \eqref{eq:kirchhoff}, 
	with $A_v$ and $B_v$ given by \eqref{eq:generalized_kirchoff}.
	Then, if $\mv{1}_3$ denotes a $3\times 3$ matrix with all elements equal to $1$ and $\mv{I}_3$ is a $3\times 3$ identity matrix,
	$Q_v = \nicefrac13 \mv{1}_3$, $P_v = \nicefrac13(3\mv{I}_3 - \mv{1}_3)$, 
	$P_{1v}=\diag(0,1,1)$, and 
	$Q_{1v}= \diag(1,0,0)$.
	Further, the operator $\Lambda_v$ is the multiplication by the number $-\nicefrac{\alpha}{3}$.	
\end{example}

We let $L$ denote the elliptic operator defined in \eqref{eq:operatorL}, equipped with vertex conditions \eqref{eq:boundaryconditionsselfadjoint} and domain
\begin{align}\label{domainL}
	D(L):=\big\{u: u \in {\widetilde{H}}^2(\Gamma)\mbox{ and}~  A_vF(v)+\H(v)B_vF^{'}(v)=0 ~~~ \forall~v\in \mathcal{V}\big\}.
\end{align}
The following assumption is used in several results throughout the paper.
\begin{Assumption}\label{assum1}	
 	Define  $S:=\displaystyle\max_{v\in \mathcal{V}}\{\|\Lambda_v\|_2\}<\infty$, where $\|\Lambda_v\|_2$ is the spectral norm of $\Lambda_v$, i.e., the operator norm of $\Lambda_v$ with respect to the Euclidean norm. Recall that $\|\Lambda_v\|_2 = \sigma_{\max}(\Lambda_v)$, where $\sigma_{\max}(A)$ is the largest singular value of  $A$. Assume that 
	$\kappa_{0}^{2}>4Sl_{min}^{-1}$ and   ${2Sl_{max}\leq \H_0}$,  
	where ${l_{min}=\displaystyle\min_{e\in \mathcal{E}}l_e}$ and $l_{max} = \displaystyle\max_{e\in\mathcal{E}} l_e$.
\end{Assumption}

\begin{Remark}\label{rmk1}
	If  $A_v=0$ or $B_v=0$ in \eqref{BCs} (i.e., the cases of Dirichlet or Neumann vertex conditions), we have $S=0$. For the generalized Kirchhoff vertex conditions \eqref{eq:kirchhoff}, $S=\nicefrac{\lvert\alpha\rvert}{{d_0}}$, where $d_0 = \displaystyle\min_{v\in\mathcal{V}} d_v$.
\end{Remark}

\begin{Remark}\label{rmk2}
	The conditions $\kappa_{0}^{2}>4Sl_{e}^{-1}$ and $2Sl_e\leq \H_0$ in Assumption~\ref{assum1} can be written as $\frac{4S}{\kappa_{0}^{2}}<l_e\leq \frac{\H_0}{2S}$. Therefore, when $S\neq 0$, the assumption is satisfied if 
	$l_{min}, l_{max}\in \left(\frac{4S}{\kappa_{0}^{2}}, \frac{\H_0}{2S}\right]$.	
\end{Remark}

Assumption \ref{assum1} above can always be replaced by the following alternative:
\begin{Assumption}\label{assum1_alt}	
	For every $v\in\mathcal{V}$, the matrix $\Lambda_v$ is non-positive definite. 
\end{Assumption}

\begin{Remark}\label{rem:assum1_alt}
	In the case of Kirchhoff vertex conditions \eqref{eq:kirchhoff}, Assumption \ref{assum1_alt} translates to requiring $\alpha\geq 0$.
\end{Remark}

\begin{Theorem}\label{compactLinverse}
	If either Assumption~\ref{assum1} or Assumption~\ref{assum1_alt} holds, then $L$ defined in \eqref{eq:operatorL} is densely defined, self-adjoint, and positive definite with a compact inverse.
\end{Theorem}
\begin{proof} 
	We will only prove assuming Assumption \ref{assum1}, as the other case is simpler. By \cite[Theorem 1.4.1.1]{grisvard}, for any $u\in \widetilde{H}^2(\Gamma)$, we have ${\H\nabla u \in \widetilde{H}^1(\Gamma)}$ and thus that $Lu$ is well-defined. Let $D(L_e) = \{u|_e: u\in D(L)\}$ be the restriction of the functions in $D(L)$ to the edge $e$. As we have $\overline{C^{\infty}_{c}(e)}^{\|\cdot\|_{L_2(e)}}=L_2(e)$, and
	$ C^{\infty}_{c}(e)\subset D(L_e)\subset {{H}}^{2}(e)\subset L_2(e)$, $L$ is densely defined on each $e\in \mathcal{E}$ and, therefore, on $L_2(\Gamma)$.
	The self-adjointness of the operator $L$ follows from a straightforward adaptation of {\cite[Theorem 5]{kuchment2008quantum}}.
	To prove its positive definiteness,  note that for any $u\in D(L)$, we have
	\begin{align}\label{positivedefinite}
		(Lu, u)&=\left(\kappa^2u-\nabla \left(\H \nabla u\right), u\right)=(\kappa^2 u, u)+\left(-\nabla \left(\H \nabla u\right), u\right)\nonumber\\
		&=(\kappa^2 u, u)+\sum_{e\in \mathcal{E}}\int_{e} \H_e(x)  (u_e'(x))^2 ~dx-\sum_{v\in \mathcal{V}}\langle \Lambda_v F(v), F(v)\rangle ~~~~~~~~~~
	\end{align}
	by following the quadratic formulation in {\cite[Theorem 9]{kuchment2004quantum}}, where $\langle \cdot, \cdot\rangle$ denotes the standard inner product in $\mathbb{R}^{d_v}$.
	The last sum in \eqref{positivedefinite} can be bounded as
	\begin{align}\label{vertexcal}
		\sum_{v\in \mathcal{V}}\langle \Lambda_v F(v), F(v)\rangle&	\leq S\sum_{v\in \mathcal{V}}\lvert F(v)\rvert^2 
		\leq 2S \sum_{e\in \mathcal{E}}\left( \frac{2}{l_e}\|u\|^{2}_{L_2(e)}+l_e\|u_e'\|^{2}_{L_2(e)}\right),
	\end{align}
	where the last inequality is obtained using {\cite[Lemma 8]{kuchment2004quantum}}.
	Substituting \eqref{vertexcal} into \eqref{positivedefinite}, we obtain
	\begin{align}\label{positivedefinite2}
		(Lu, u)	& \geq \sum_{e\in\mathcal{E}}\kappa_{0}^{2}\|u\|^{2}_{L_2(e)}+ \sum_{e\in\mathcal{E}}\H_0\|u_e'\|^{2}_{L_2(e)}-2S \sum_{e\in\mathcal{E}}\big( \frac{2}{l_e}\|u\|^{2}_{L_2(e)}+l_e\|u_e'\|^{2}_{L_2(e)}\big)\nonumber\\
		&\geq \sum_{e\in\mathcal{E}}(\kappa_{0}^{2}-4Sl_{e}^{-1})\|u\|^{2}_{L_2(e)}+
		\sum_{e\in\mathcal{E}}(\H_0-2Sl_e)\|u_e'\|^{2}_{L_2(e)}.
	\end{align}

	If $\kappa_{0}^{2}>4Sl_{min}^{-1}$ and $2Sl_{max}\leq \H_0$, then	$(Lu,u)>0, \forall u\neq 0$, 
	which shows that $L$ is positive definite. 
	As $L$ is self-adjoint, it is closed (e.g., see \cite[Chapter VII, Section 3, Proposition 2]{Yosida}). As $L$ is positive definite, it follows from \cite[Chapter VII, Section 5, Corollaries 1 and 2]{Yosida} that $L$ has a continuous inverse $L^{-1}$. 
	
	Finally,  the inclusion  mapping $I:D(L)\subset L_2(\Gamma)\rightarrow L_2(\Gamma),$ is compact by the Rellich--Kondrachov compactness theorem (see \cite{evans2010partial}). Because ${L^{-1}}$ is continuous, the composition $(IL^{-1})$ is a compact mapping. That is, ${L^{-1}}$ is compact.
\end{proof}

\begin{Remark}\label{rmk3}
	If $S=0$ in Theorem~\ref{compactLinverse}, then $L$ is always positive definite, and by Remark~\ref{rmk2}, Theorem~\ref{compactLinverse} does not depend on any edge length.
\end{Remark}

As a consequence of Theorem~\ref{compactLinverse}, there exists a complete orthonormal system   $\{e_k\}_{k\in \mathbb{N}}$ on $L_2(\Gamma)$, which diagonalizes the operator $L$, and a set of eigenvalues $\{\lambda_k\}_{k\in \mathbb{N}}$ corresponding to 
$\{e_k\}_{k\in \mathbb{N}}$. We will now study in detail the behavior of the sequence of eigenvalues $\{\lambda_k\}_{k\in \mathbb{N}}$.

\subsection{Weyl's law for the differential operator $L$}

The following results follow from direct adaptations of some results of \cite{berkolaikodependence, berkolaiko2017elementary}, and 
\cite{Berkolaiko2013}. We provide the details for completeness. 
We begin by recalling that we are interested in vertex conditions of the form \eqref{BCs}. For each $v\in\mathcal{V}$, we take a pair of matrices $A_v$ and $B_v$ and recall that $\mathcal{A}_v = (A_v,\, B_v)$ is the horizontal concatenation of $A_v$ and $B_v$. The collection of these matrices across $\mathcal{V}$ is denoted by
$\mathcal{A} = \{\mathcal{A}_v\}_{v\in\mathcal{V}}$. We regard $\mathcal{A}$ as a block diagonal matrix of size $\sum_v d_v \times 2 \sum_v d_v$, where each block has size $d_v \times 2d_v$. We identify the space of all matrices $\mathcal{A}$ with $\mathbb{C}^{2\sum_v d_v^2}$. Let $U$ be the subset of $\mathbb{C}^{2\sum_v d_v^2}$ consisting of the matrices that satisfy the maximal rank condition:
$
U = \left\{\mathcal{A}\in \mathbb{C}^{2\sum_v d_v^2}: \mathcal{A}_v \text{ has maximal rank for any }v \in \mathcal{V}\right\}
$
and $U_s$ be the subset of $U$ consisting of the matrices satisfying the self-adjointness condition 
$U_s = \left\{\mathcal{A}\in U: A_vB_v^\ast \text{ is self-adjoint for any }v \in \mathcal{V}\right\}.$

The operator given by \eqref{eq:operatorL} is denoted by $L_\mathcal{A}$, with domain $D(L_\mathcal{A})$ given by \eqref{domainL}, where $\mathcal{A}=(A_v\, B_v)$. We observe that under more general conditions than those in Theorem~\ref{compactLinverse} the spectrum of $L_{\mathcal{A}}$ is discrete. Indeed, we have that the operator $L$, with $\kappa\geq 0$ (in particular we can take $L=-\Delta$), under the vertex conditions \eqref{eq:boundaryconditionsselfadjoint} has discrete spectrum. To this end, we follow \cite[Theorem~3.1.1]{Berkolaiko2013}.

\begin{Theorem}\label{thm:spectrumdiscreteL_A}
	Let $L$ be given by \eqref{eq:operatorL}, such that $\kappa\in L_\infty(\Gamma)$ and $H$ is Lipschitz and positive. For each $\mathcal{A}\in U_s$, the resolvent $(L_{\mathcal{A}} - i I)^{-1}$ is a compact operator in $L_2(\Gamma)$. In particular, the spectrum of $L_{\mathcal{A}}$ consists exclusively of isolated eigenvalues $\lambda_j(L_{\mathcal{A}})$ of finite multiplicity with $\lambda_j \to \infty$ as $j\to \infty$.
\end{Theorem}

\begin{proof}
	As $L_{\mathcal{A}}$ is self-adjoint, the resolvent $(L_{\mathcal{A}} - i I)^{-1}$ exists and continuously maps $L_2(\Gamma)$ into $D(L_{\mathcal{A}}) \subset \widetilde{H}^2(\Gamma)$. By the Rellich--Kondrachov compactness theorem \cite[see]{evans2010partial}, the resolvent is compact, which proves the claims about the spectrum and the compactness of the resolvent.
\end{proof}

We can now obtain an eigenvalue interlacing inequality for $L_{\mathcal{A}}$, where ${\mathcal{A}\in U_s}$. By following \cite{berkolaikodependence}, we examine the effect of modifying the vertex condition at a single vertex $v\in\mathcal{V}$. We let $\Gamma_v$ be a compact (not necessarily connected) metric graph with a distinguished vertex $v$. We assume that all vertices $w\in\mathcal{V}$, where $w\neq v$, satisfy some self-adjoint vertex conditions, that is, vertex conditions such that the corresponding operator is self-adjoint, whereas $v$ is endowed with the following condition with coefficient $\alpha$:
\begin{equation}\label{eq:conditions_alpha}
	\{u\hbox{ is continuous at $v$ and $\H(v) \sum_e \partial_e u(v) = \alpha u(v)$}
	\},
\end{equation}
or, 
\begin{equation}\label{eq:conditions_alpha2}
	\{u\hbox{ is continuous at $v$ and $\frac{\H(v)}{\alpha} \sum_e \partial_e u(v) = u(v)$}\},
\end{equation}
for  $\alpha\neq 0$. Note that taking $\alpha\to \infty$ in \eqref{eq:conditions_alpha2} gives Dirichlet vertex conditions. 
By Theorem~\ref{thm:spectrumdiscreteL_A}, the spectrum of $L_{\mathcal{A}}$ consists of isolated eigenvalues that tend to infinity, and we assume that they are ordered so that: ${\lambda_1(L_{\mathcal{A}})\leq \lambda_2(L_{\mathcal{A}}) \leq \cdots}$.

\begin{Theorem}\label{thm:interlacing}
	Let $\mathcal{A}_\alpha$ be the matrix associated with the vertex condition \eqref{eq:conditions_alpha} (or, equivalently, to \eqref{eq:conditions_alpha2}), with arbitrary self-adjoint conditions at vertices $w\neq v$. Further, let $\mathcal{A}_{\widetilde{\alpha}}$ be the matrix obtained by changing the condition at vertex $v$ from $\alpha$ to $\widetilde{\alpha}$. If $-\infty < \alpha < \widetilde{\alpha} \leq \infty$, where, by \eqref{eq:conditions_alpha2}, $\widetilde{\alpha}=\infty$ corresponds to the Dirichlet condition, then
	$\lambda_n(L_{\mathcal{A}_\alpha}) \leq \lambda_n(L_{\mathcal{A}_{\widetilde{\alpha}}}) \leq \lambda_{n+1} (L_{\mathcal{A}_\alpha}).$
\end{Theorem}

\begin{proof}
	The proof follows directly from the expression of the bilinear form \eqref{positivedefinite} and identical arguments, as in the proof of \cite[Theorem 5.1]{berkolaikodependence}.
\end{proof}

As a corollary of this result, we obtain a Weyl's law for the eigenvalues in case of generalized Kirchhoff vertex conditions.
\begin{Corollary}\label{cor:genkirchhoffweyl}
	Let $L$ be the second-order elliptic operator given by \eqref{eq:operatorL}, 
	with coefficients satisfying \eqref{eq:condcoeff1} and \eqref{eq:condcoeff2}, 
	endowed with generalized Kirchhoff vertex conditions \eqref{eq:kirchhoff}.
	Further, either let Assumption \ref{assum1_alt} hold, or let Assumption~\ref{assum1} hold with 
	$S = \frac{\lvert \alpha\rvert}{d_0}$ (i.e., $S$ for generalized Kirchhoff vertex conditions). 	
	Then, constants $C_1,C_2>0$ exist such that
	\begin{equation}\label{eq:weylgenkir}
		\forall n\in\mathbb{N}: \qquad C_1 n^2 \leq \lambda_{n} \leq C_2 n^2,
	\end{equation}
	where $\{\lambda_k\}_{k\in\mathbb{N}}$ are the eigenvalues of $L$ ordered in increasing order.
	
\end{Corollary}

\begin{proof}
	By the same arguments of the proof of 
	\cite[Lemma 4.4]{berkolaiko2017elementary}, we can use the interlacing 
	inequality (Theorem~\ref{thm:interlacing}) to demonstrate that the result 
	follows from a corresponding Weyl's law for the operator $L_D$, 
	given by \eqref{eq:operatorL} endowed with Dirichlet vertex conditions. 
	Next, when we endow $L_D$ with Dirichlet vertex conditions, 
	the operator becomes completely decoupled across the edges, meaning 
	that we can regard $\Gamma$ as a disjoint union of intervals, and 
	$L_D$ as the operator of \eqref{eq:operatorL} and Dirichlet 
	boundary conditions on each interval. As $\Gamma$ is compact, it 
	has finitely many edges, each with a finite length. Therefore, it is 
	sufficient to demonstrate that, for each interval $J$ with a finite length, the 
	operator $L$ given by \eqref{eq:operatorL} acting on 
	$H^2(J)\cap H^1_0(J)$ (i.e., it satisfies Dirichlet boundary 
	conditions on $J$) obeys Weyl's law, which holds by 
	\cite[Theorem 6.3.1]{davies1996spectral}.
\end{proof}

We end this section by proving that $L$ satisfies a Weyl's law under general vertex conditions. However, compared to the case of generalized Kirchhoff vertex conditions, we in this case will require an additional assumption which we now introduce. Given a family of matrices $(C_v)_{v\in\mathcal{V}}$, $\diag(C_v)_{v\in \mathcal{V}}$ denotes the block diagonal matrix with diagonal entries given by $C_v$, $v\in \mathcal{V}$. If
$(A_v,B_v)_{v\in\mathcal{V}}$ is a collection of matrices such that
for every $v\in \mathcal{V}$, $(A_v,B_v)$ has maximal rank and $A_vB_v^\ast    $
is self-adjoint, then $(A,B)$ has maximal rank and $AB^\ast$ is self-adjoint,
where $A = \diag(A_v)_{v\in\mathcal{V}}$ and ${{B}=\diag(\H(v)B_v)_{v\in\mathcal{V}}}$. We can similarly define $P = \ker B$ and $Q = I-P$.
Next, we let $\mathcal{L} = \left( {B}\vert_{\textrm{ran} {B}^\ast}\right)^{-1} A Q$, and $\sigma(\mathcal{L})$ be the spectrum of $\mathcal{L}$. We define 
$\lambda_{\min}^+ = \inf\{ \lambda: \lambda>0\hbox{ and }\lambda \in \sigma(\mathcal{L})\},$
with the convention that $\inf \emptyset = +\infty$. Following 
\cite{Odzak2019Weyl}, we define 
\begin{equation}\label{eq:lfunctionweyl}
	\ell(s) = \frac{\log(2\vert \mathcal{E}\vert)}{s} + \frac{2}{s} \textrm{arctanh}\left(\frac{s}{\lambda_{\min}^+}\right).
\end{equation}
We let $\sigma=+\infty$ if $\lambda_{\min}^+ = +\infty$, and let $\sigma \in (0,\lambda_{\min}^+)$ be the point in which $l$ attains its unique minimum if $\lambda_{\min}^+$ is finite.

\begin{Theorem}\label{cor:weyllaw}
	Let $L$ be the second-order elliptic operator given by \eqref{eq:operatorL}, 
	with coefficients satisfying \eqref{eq:condcoeff1} and \eqref{eq:condcoeff2}, 
	endowed with vertex conditions \eqref{BCs} and satisfying Assumption~\ref{assum1} or \ref{assum1_alt}. Assume that for every $e\in\mathcal{E}$, $l_e > \ell(\sigma)$, where
	$\ell$ is given in \eqref{eq:lfunctionweyl}. 
	Then, constants ${C_1,C_2>0}$ exist such that
	\begin{equation}\label{eq:eiga}
		\forall n\in\mathbb{N}: \qquad C_1 n^2 \leq \lambda_{n} \leq C_2 n^2,
	\end{equation}
	where $\{\lambda_k\}_{k\in\mathbb{N}}$ are the eigenvalues of $L$ ordered in increasing order.
\end{Theorem}

\begin{proof}
	The goal is to obtain a Weyl's law for $L$ using
	Weyl's law from the Laplacian given in \cite[Theorem 2]{Odzak2019Weyl}. Observe that the existence of the eigendecomposition for the Laplacian follows from Theorem \ref{thm:spectrumdiscreteL_A}, whereas the Weyl's law for its eigenvalues follows from \cite[Theorem 2]{Odzak2019Weyl}.
	To this end, we let $\widetilde{L} = \kappa_0^2 - \H_0\Delta$ with vertex 
	conditions $A_vF(v) + \H(v) B_v F'(v) = 0, v\in\mathcal{V}$. 
	It is apparent that the same proof of Theorem~\ref{compactLinverse}
	also works for the operator $\widetilde{L}$, with the same bounds on
	the lengths of the edges. Let $\lambda_1^{\widetilde{L}}\leq \lambda_2^{\widetilde{L}} \leq \cdots$ and 
	$\lambda_1^{-\Delta}\leq \lambda_2^{-\Delta} \leq \cdots$
	be the eigenvalues of $\widetilde{L}$ and $-\Delta$ with the same vertex conditions, respectively. Then,
	\begin{equation}\label{eq:compeigenval}
		\lambda_j^{\widetilde{L}} = \kappa_0^2 + \H_0 \lambda_j^{-\Delta}.
	\end{equation}
	Next, we let $A = \diag(A_v)_{v\in\mathcal{V}}$ and $\widetilde{B} = \diag(\H_0 B_v)_{v\in\mathcal{V}}$. Because $\H_0>0$, $(A_v,B_v)$ has full rank and $A_vB_v^\ast$ is self-adjoint, for every $v\in\mathcal{V}$, it follows that $(A,\widetilde{B})$ has full rank and $A\widetilde{B}^\ast$ is self-adjoint. Furthermore, 	by assumption, $\min_{e\in\mathcal{E}} l_e > \ell(\sigma)$. Hence, by \cite[Theorem 2]{Odzak2019Weyl}, the eigenvalues $(\lambda_j^{-\Delta})_{j\in\mathbb{N}}$ satisfy Weyl's law. From \eqref{eq:compeigenval}, it follows that the eigenvalues of $\widetilde{L}$, $(\lambda_j^{\widetilde{L}})_{j\in\mathbb{N}}$, satisfy Weyl's law.
	
	We compare the eigenvalues of $L$ to the eigenvalues of $\widetilde{L}$
	to obtain a Weyl's law for $(\lambda_j)_{j\in\mathbb{N}}$, the eigenvalues of $L$.
	By \eqref{positivedefinite}, we have
	\begin{equation}\label{eq:dirichletformLtilde}
		(\widetilde{L}u,u) = (\kappa_0^2 u, u)+\sum_{e\in\mathcal{E}}\int_{e} \H_0 (u_e'(x))^2\,dx-\sum_{v\in\mathcal{V}} \frac{\H_0}{\H(v)}\langle \Lambda_v F(v), F(v)\rangle.
	\end{equation}
	Moreover, $\kappa$ and $\H$ are bounded ($\H$ is continuous and $\Gamma$ is compact). Therefore, by \eqref{positivedefinite}, \eqref{eq:dirichletformLtilde}, and
	\eqref{vertexcal} and the assumptions on the lengths of the edges, constants $c_1, C_1, c_2,C_2>0$ exist such that
	$$c_1 \|u\|_{H^1(\Gamma)}^2 \leq (Lu, u) \leq C_1 \|u\|_{H^1(\Gamma)}^2, \quad c_2 \|u\|_{H^1(\Gamma)}^2 \leq (\widetilde{L}u, u) \leq C_2 \|u\|_{H^1(\Gamma)}^2.$$
	In particular,
	$\frac{c_1}{C_2} (\widetilde{L}u, u) \leq (Lu,u) \leq \frac{C_1}{c_2} (\widetilde{L}u,u).$
	We divide all terms in these inequalities by $\|u\|_{L_2(\Gamma)}^2$
	and use the min-max principle to demonstrate that, for every $j\in\mathbb{N}$,
	$\frac{c_1}{C_2} \lambda_j^{\widetilde{L}} \leq \lambda_j \leq \frac{C_1}{c_2} \lambda_j^{\widetilde{L}}.$
	Finally, the eigenvalues $(\lambda_j^{\widetilde{L}})_{j\in\mathbb{N}}$
	satisfy Weyl's law; thus, it follows that $\theta,\Theta>0$ exist such that
	$\theta \frac{c_1}{C_2} j^2 \leq \lambda_j \leq \Theta \frac{C_1}{c_2} j^2.$
\end{proof}

\section{Existence and uniqueness of solutions}\label{sec:solutions}
Let us begin by introducing various fractional-order Sobolev spaces that we will need, and which are summarized in Table~\ref{tab:fractionalspaces}. To do so, recall the real interpolation of Hilbert spaces, e.g., \cite[Appendix~A]{BSW2022}. For two Hilbert spaces $H_1\subset H_0$, we let $(H_0,H_1)_s$ denote the real interpolation of order $0<s<1$ between $H_0$ and $H_1$, and let $(\cdot,\cdot)_{(H_0,H_1)_s}$ denote its induced inner product. 
For $e\in\mathcal{E}$ and $s>0$, the fractional Sobolev space $H^s(e)$ is defined in the usual way as for an interval  \cite[p.77]{mclean}. 
Similarly, given a system $\{e_k,\lambda_k\}$ of eigenvectors and eigenvalues of $L$, the fractional operator $L^\beta$ for ${\beta>0}$ is defined in the spectral sense, i.e., the action of $L^\beta$ on $\phi\in D(L^\beta)$ is $L^\beta\phi = \sum_{k\in\mathbb{N}}\lambda_k^\beta (\phi, e_k)e_k$, where the domain is $ D(L^\beta) = \dot{H}_L^{2\beta}(\Gamma)$, which is a Hilbert space equipped with its natural inner product, and its dual is denoted by $\dot{H}_L^{-2\beta}(\Gamma)$ (see, e.g., \cite{BKK2020}). 

\begin{table}[b]
\begin{tabular}{ |p{2.4cm}|p{4.4cm}|p{4.7cm}|  }
	\hline
	\multicolumn{3}{|c|}{Fractional Sobolev-type spaces} \\
	\hline
	Notation& Definition & (Hilbertian) norms\\
	\hline
	$H^s(e), e\in\mathcal{E}$ & See \cite[p.77]{mclean} & $\|\cdot\|_{H^s(e)}$ (see \cite[p.77]{mclean})\\
	$H_0^s(e), e\in\mathcal{E}$ & Closure of $C_c^\infty(e)$ in $H^s(e)$ & $\|\cdot\|_{H^s(e)}$\\ 
	$H^{-s}(e), e\in\mathcal{E}$ & Dual of $H_0^s(e)$ & $\|\cdot\|_{H^{-s}(e)}.$ Dual norm\\
		$H^s(\Gamma),0\!<\!s\!<\!1$ &
	$(L_2(\Gamma),H^1(\Gamma))_s$ & $\|\cdot\|_{H^s(\Gamma)}:= \|\cdot\|_{(L_2(\Gamma), H^1(\Gamma))_s}$\\		
	${H}^s(\Gamma),1\!<\!s\!<\!2$ &
	$(H^1(\Gamma),\widetilde{H}^2_C(\Gamma))_s$ & $\|\cdot\|_{{H}^s(\Gamma)}:= \|\cdot\|_{(H^1(\Gamma), \widetilde{H}_C^2(\Gamma))_s}$\\
	$\widetilde{H}^s(\Gamma),0\!<\!s\!<\!1$ &
								$(L_2(\Gamma),\widetilde{H}^1(\Gamma))_s$ & $\|\cdot\|_{\widetilde{H}^s(\Gamma)}:=\|\cdot\|_{(L_2(\Gamma), \widetilde{H}^1(\Gamma))_s}$\\
	$\widetilde{H}^s(\Gamma),1\!<\!s\!<\!2$ &
								$(\widetilde{H}^1(\Gamma),\widetilde{H}^2(\Gamma))_s$ & $\|\cdot\|_{\widetilde{H}^s(\Gamma)}:= \|\cdot\|_{(\widetilde{H}^1(\Gamma), \widetilde{H}^2(\Gamma))_s}$\\		
	$\dot{H}_L^s(\Gamma), s\geq0$ & $\left\{u\in L_2(\Gamma)\!:\!\|u\|_{\dot{H}^s(\Gamma)}\!\!<\!\infty\right\}\!$ & $\|\cdot\|_{\dot{H}^s(\Gamma)}^2\!\!:=\!\!\sum_{k\in \mathbb{N}_0}\lambda_{k}^{s}(\cdot,e_k)^{2}$\\
	$\dot{H}^{-s}_L(\Gamma), s >0$& Dual of $\dot{H}_L^s(\Gamma)$ & $\|\cdot\|_{\dot{H}_L^{-s}(\Gamma)}.$ Dual norm\\	
	\hline 
   \end{tabular}
   \caption{Fractional Sobolev spaces on $\Gamma$. Recall that  
   	$(e_k)_{k\in\mathbb{N}}$ are the eigenvectors of $L$ and that  $\widetilde{H}_C^r(\Gamma) = \widetilde{H}^r(\Gamma)\cap C(\Gamma)$, for $r\in\mathbb{N}$.
   }
\label{tab:fractionalspaces}
\end{table}

\begin{Proposition}\label{prp:existuniqdeter}
	Suppose that the operator in \eqref{Amaineqn_deterministic} is equipped with 
	the vertex conditions \eqref{eq:boundaryconditionsselfadjoint} 
	and that the coefficients satisfy \eqref{eq:condcoeff1} and 
	\eqref{eq:condcoeff2}. 
	If either Assumption \ref{assum1} or Assumption \ref{assum1_alt} holds, then, for any $f\in L_2(\Gamma)$ and any $\beta>0$, 
	a unique solution $u \in \dot{H}_L^{2\beta}(\Gamma)$ 
	to \eqref{Amaineqn_deterministic} exists.
\end{Proposition}
\begin{proof}
	We combine Theorem~\ref{compactLinverse} with  
	{\cite[Lemma 2.1]{BKK2020}} to demonstrate the 
	unique continuous extension of $L^\beta$ to an isometric isomorphism $L^\beta:\dot{H}_L^s(\Gamma)\rightarrow \dot{H}_L^{s-2\beta}(\Gamma),$ 
	for any $s\in \mathbb{R}$. The result follows by taking $s=2\beta$.
\end{proof}

To show existence and uniqueness of the solution to \eqref{Amaineqn}, observe that the Gaussian white noise $\mathcal{W}$ on $L_2(\Gamma)$ can be represented by a series
	$\mathcal{W}=\sum_{k\in \mathbb{N}} \xi_ke_k,$
where $\{\xi\}_k$ is a sequence of independent, real-valued, standard normally distributed random variables on $(\Omega, \mathcal{F},\mathbb{P})$.

\begin{Proposition}\label{maintheorem}
	Suppose that the operator in \eqref{Amaineqn} is equipped with the vertex conditions \eqref{eq:boundaryconditionsselfadjoint} and that the coefficients satisfy \eqref{eq:condcoeff1} and \eqref{eq:condcoeff2}. If either Assumption \ref{assum1} or Assumption \ref{assum1_alt} holds, then, for any $\beta>0$ and any $\epsilon>0$, a unique solution $u\in L_2(\Omega, \dot{H}_L^{2\beta-\frac{1}{2}-\epsilon}(\Gamma))$ exists to \eqref{Amaineqn}.
	Further, if $\beta>\frac{1}{4}$, $u\in L_2(\Omega, L_2(\Gamma))$.
\end{Proposition}
\begin{proof}
	By {\cite[Proposition 2.3] {BKK2020}} and
	Theorem~\ref{cor:weyllaw}, 
	${\mathcal{W}\in L_2(\Omega, \dot{H}_L^{-\frac{1}{2}-\epsilon}(\Gamma))}$ 
	holds for all $\epsilon>0$.
	Furthermore, also by {\cite[Proposition 2.3]{BKK2020}},
	a constant $C_\lambda>0$, depending only on $\beta$ and on the constant $C_1$ in Theorem~\ref{cor:weyllaw}, exists
	such that
	\begin{align}\label{noiseregularity}
		\mathbb{E}\left[\|\mathcal{W}\|^{2}_{-\frac{1}{2}-\epsilon}\right]\leq C_{\lambda}^{\frac{1}{2}-\epsilon}\left(1+\frac{1}{2 \epsilon}\right).
	\end{align}	

	By combining Theorem~\ref{compactLinverse} with  
	{\cite[Lemma 2.1]{BKK2020}} and the noise regularity \eqref{noiseregularity}, the result follows.
\end{proof}

An alternative representation of Gaussian white noise is as a family of centered Gaussian variables $\{\mathcal{W}(h): h\in L_2(\Gamma)\}$, which satisfy
\begin{equation}\label{isometryW}
	\mathbb{E}[\mathcal{W}(h)\mathcal{W}(g)] = (h,g) \qquad \forall h,g\in L_2(\Gamma).
\end{equation}
With this representation, one can observe that $u$, the unique solution of \eqref{Amaineqn}, for $\beta>\frac{1}{4}$, is a Gaussian random field satisfying
\begin{equation}\label{eq:solspde}
	(u,\psi) = \mathcal{W}(L^{-\beta}\psi) \quad \mathbb{P}\text{-a.s.} \quad \forall \psi\in L_2(\Gamma),
\end{equation}
with covariance operator $\mathcal{C} = L^{-2\beta}$ satisfying
$$
(\mathcal{C}\phi,\psi) = \mathbb{E}[(u,\phi)(u,\psi)] \quad \forall \phi,\psi \in L_2(\Gamma).
$$ 
Further, $u$ is centered, and we let $\varrho$ denote the covariance function corresponding to $\mathcal{C}$, defined as 
\begin{equation}\label{eq:covfunc}
	\varrho(s,s') = \mathbb{E}(u(s)u(s')) \quad \text{a.e. in $\Gamma\times \Gamma$}.
\end{equation}
Observe that the expression in \eqref{eq:covfunc} exists for $\beta>1/4$. Indeed, we have that in this case, by Proposition~\ref{maintheorem}, $u\in L_2(\Omega,L_2(\Gamma))$, so by Fubini's theorem and Cauchy-Schwarz inequality $\lvert \mathbb{E}(u(s)u(s'))\rvert <\infty$ a.e.~in $\Gamma\times \Gamma$. 
We discuss the covariance function in more detail in the next section.

\section{Regularity in the Kirchhoff case}\label{sec:regularity}

We let  $L_\alpha$ denote the operator \eqref{eq:operatorL} endowed with the generalized Kirchhoff vertex conditions, with coefficients satisfying \eqref{eq:condcoeff1} and \eqref{eq:condcoeff2}. 
Further let ${\lambda_1\leq \lambda_2\leq \lambda_3\leq \cdots}$ be its corresponding 
eigenvalues written in increasing order, with corresponding eigenvectors
$\{e_k\}_{k\in\mathbb{N}_0}$. 
In this section, we state the regularity of the solutions to \eqref{Amaineqn_deterministic} and \eqref{Amaineqn} in the case of Kirchhoff vertex conditions ($\alpha=0$). To do so, we need the following characterization of the spaces $\dot{H}^\beta_{L_0}(\Gamma)$ for $\beta\in [0,2]$. 

\begin{Theorem}\label{thm:characterization}
	Let $0<\beta<2$, with $\beta\neq \nicefrac32$. 
	Then, $\dot{H}^{\nicefrac{1}{2}}_{L_0}(\Gamma)\cong H^{\nicefrac{1}{2}}(\Gamma)$ and 
	$$\dot{H}_{L_0}^\beta(\Gamma) \cong \begin{cases}
		H^\beta(\Gamma) \cong \widetilde{H}^\beta(\Gamma),& \hbox{if } 0<\beta<\nicefrac12,\\
		H^\beta(\Gamma) \cong \widetilde{H}^\beta(\Gamma) \cap C(\Gamma), & \hbox{if } \nicefrac12<\beta<\nicefrac32,\\
		H^\beta(\Gamma) \cap K^\beta(\Gamma) \cong \widetilde{H}^\beta(\Gamma) \cap C(\Gamma) \cap K^\beta(\Gamma),& \hbox{if } \nicefrac32<\beta\leq 2,
	\end{cases}$$
	where
	$K^\beta(\Gamma) = \Bigl\{u\in\widetilde{H}^\beta(\Gamma): \forall v\in\mathcal{V}, \sum_{e\in \mathcal{E}_v} \partial_e u(v) = 0 \Bigr\}.$ 
	Furthermore, if $\beta>\nicefrac12$ and we define $\widetilde{\beta} = \beta - \nicefrac12$ if $\beta\leq 1$ and $\widetilde{\beta}=\nicefrac12$ if $\beta>1$, then 
	\begin{equation}\label{eq:sobembed}
		\dot{H}^\beta_{L_0}(\Gamma) \hookrightarrow C^{0,\widetilde{\beta}}(\Gamma).
	\end{equation}
\end{Theorem}
\begin{Remark}
	The equalities inside the brackets in the statement of Theorem~\ref{thm:characterization}, which is one of our main results,  are also part of the statement. 
		In \cite{BSW2022}, $\dot{H}_{\kappa^2 - \Delta}^\beta(\Gamma)$, for $0 \leq \beta \leq 1$ or $\beta\in\mathbb{N}$ and $\kappa>0$ constant, is characterized in terms of $H^\beta(\Gamma)$, but not in terms of $\widetilde{H}^\beta(\Gamma)$ except if $\beta\in\mathbb{N}$, which is the most difficult part to prove for fractional $\beta$. Thus, Theorem~\ref{thm:characterization} is a substantial extension of the corresponding results in  \cite{BSW2022}, also to more general operators with non-constant coefficients. 
 \end{Remark}
The proof is postponed until the next section to simplify the exposition. Before that, we use it to derive some regularity results for the solutions to \eqref{Amaineqn_deterministic} and \eqref{Amaineqn}.

\begin{Theorem}\label{thm:regularity0}
	Let $u$ be the solution of 
	\eqref{Amaineqn_deterministic}, where $L$ is 
	equipped with Kirchhoff vertex conditions and the 
	coefficients satisfy \eqref{eq:condcoeff2} and \eqref{eq:condcoeff1}.
	Then ${u\in H^1(\Gamma)}$ if and only if $\beta>\nicefrac12$.
	If $\beta\in (\nicefrac14, \nicefrac12]$, 
	then $u$ is $\gamma$-H\"older continuous,
	with $\gamma=2\beta - \nicefrac12$.
\end{Theorem}
\begin{proof}
	By Proposition~\ref{prp:existuniqdeter}, $u\in \dot{H}_{L_0}^{2\beta}(\Gamma)$.
	Therefore, all the statements follow directly from Theorem~\ref{thm:characterization}. In particular,
	the H\"older-continuity follows from \eqref{eq:sobembed}.
\end{proof}

The following result can be considered an extension of \cite[Theorem 1]{BSW2022} to the case of non-constant coefficients $\kappa,\H$ of the operator.

\begin{Theorem}\label{thm:regularity1}
	Let $u$ be the solution of \eqref{Amaineqn}, where the operator is equipped with Kirchhoff vertex conditions and the coefficients satisfy \eqref{eq:condcoeff1} and \eqref{eq:condcoeff2}.
	Then ${u\in H^1(\Gamma)}$ $\mathbb{P}$-a.s. if and only if $\beta > \nicefrac{3}{4}$. Further, for any $\varepsilon>0$ and any $1/4<\beta\leq 2+1/4$, the sample paths of $u$ belong to $\widetilde{H}^{2\beta - \nicefrac{1}{2}-\varepsilon}(\Gamma)$ $\mathbb{P}$-a.s.. Furthermore, for $\beta > 1$, $u$ satisfies the Kirchhoff's vertex conditions: $\sum_{e\in\mathcal{E}_v} \partial_e u(v) = 0$. 
\end{Theorem}

\begin{proof}
	By Proposition \ref{maintheorem}, for any $\varepsilon>0$, $u$ has sample paths in $\dot{H}^{2\beta -\frac{1}{2}-\varepsilon}_{L_0}(\Gamma)$.  Thus, if ${\beta>\nicefrac34}$, we have $\dot{H}_{L_0}^{2\beta-\frac{1}{2} -\varepsilon} (\Gamma) \subset H^1(\Gamma)$ for sufficiently small $\varepsilon>0$ The converse for $\beta \in (\nicefrac14, \nicefrac34]$ follows from Corollary~\ref{cor:genkirchhoffweyl} and the same arguments as in the proof of \cite[Theorem~1]{BSW2022}. 
	The second claim follows from $\dot{H}_{L_0}^{2\beta-\frac{1}{2} -\varepsilon} (\Gamma) \hookrightarrow \widetilde{H}^{2\beta - \nicefrac{1}{2}-\varepsilon}(\Gamma)$, which holds by Theorem~\ref{thm:characterization}. 
	Finally, by Theorem~\ref{thm:characterization}, $\dot{H}_{L_0}^{2\beta-\frac{1}{2} -\varepsilon} (\Gamma) \subset K^{2\beta-\frac{1}{2} -\varepsilon}(\Gamma)$
	for $\beta>1$ and for sufficiently small $\varepsilon>0$ such that $2\beta -\frac{1}{2}-\varepsilon > 3/2$. 
\end{proof}

\begin{Remark}
\cite[Theorem 1]{BSW2022} shows that the solution to $(\kappa^2 - \Delta)^{\beta} u =\mathcal{W}$ on $\Gamma$, where $\kappa$ is constant, satisfies Kirchhoff's vertex condition, $\sum_{e\in\mathcal{E}_v} \partial_e u(v) = 0$, for $\beta > 5/4$. Therefore, Theorem \ref{thm:regularity1} provides a considerable refinement to \cite[Theorem 1]{BSW2022} with respect to understanding for which values of $\beta$, the solution $u(\cdot)$ starts to satisfy the vertex conditions.
\end{Remark}

Next, we can obtain a more refined regularity result, which is a more general version of \cite[Lemmas 1 and 2]{BSW2022}.

\begin{Lemma}\label{lem:regularity2}
	Fix $\beta > \nicefrac14$ and let $\widetilde{\beta} = \beta-\nicefrac14$ if $\beta\leq \nicefrac12$ and $\widetilde{\beta}=\nicefrac14$ if $\beta>\nicefrac12$. For $x\in\Gamma$, let $u_0(x) = \mathcal{W}\left((L_0^{-\beta})^\ast(\delta_x)\right)$. Then,
	\begin{equation}\label{holdercontvar}
		E\left(\lvert u_0(x) - u_0(y)\rvert^2\right) \leq \|L_0^{-\beta}\|_{\mathcal{L}(L_2(\Gamma),C^{0,\widetilde{\beta}}(\Gamma))}^2 d(x,y)^{4\widetilde{\beta}}.    
	\end{equation}
	Furthermore, for any $0<\gamma<2\widetilde{\beta}$, $u_0$ has a $\gamma$-H\"older continuous modification.  Moreover, if $u$ is any $\gamma$-H\"older modification of $u_0$, then, for all $\phi\in L_2(\Gamma)$, we have $(u,\phi) = \mathcal{W}(L_0^{-\beta} \phi)$. In particular, $u$ is a solution to \eqref{Amaineqn}. That is, the solution $u$ of \eqref{Amaineqn} admits a modification with $\gamma$-H\"older continuous sample paths.
\end{Lemma}

A direct consequence is the continuity of the covariance function.

\begin{Corollary}\label{cor:covfunccont}
	Assume that the operator is equipped with Kirchhoff vertex conditions. Then, for any $\beta > \nicefrac{1}{4}$, the covariance function $\varrho(\cdot,\cdot)$ in \eqref{eq:covfunc} is  continuous. Further, the covariance operator $\mathcal{C}$ is an integral operator with kernel $\varrho$, that is, 
	$$
	(\mathcal{C} \phi)(y) = \sum_{e\in\mathcal{E}} \int_e \varrho(x,y) \phi(x) dx.
	$$
	In addition, $\varrho$ admits the following series representation in terms of the eigenvalues and eigenvectors of $L$:
	$\varrho(x,y) = \sum_{j\in\mathbb{N}} \lambda_j^{-2\beta} e_j(x) e_j(y),$
	where the convergence of the series is absolute and uniform. Moreover, the series also converges in $L_2(\Gamma\times \Gamma)$.
\end{Corollary}

The proofs of Lemma~\ref{lem:regularity2} and Corollary~\ref{cor:covfunccont} are postponed to Section~\ref{sec:proof}. 
Finally, we have a representation of the solution $u$ through a Karhunen--Lo\`eve expansion. The proof of the result is essentially identical to those of \cite[Propositions 7 and 8]{BSW2022} and we therefore omit it. 

\begin{Proposition}\label{prp:KLsolspde}
	Fix $\beta > \nicefrac{1}{4}$ and let $u$ be the solution of \eqref{Amaineqn} under Kirchhoff vertex conditions, where the coefficients satisfy \eqref{eq:condcoeff1} and \eqref{eq:condcoeff2}. Then, $u$ can be represented as 
	$
	{u(s) = \sum_{j=1}^{\infty} \xi_j \lambda_j^{-\beta} e_j(s)}
	$,
	where $\xi_j$ denotes independent standard Gaussian variables. Further, the series 
		$u_n(s) = \sum_{j=1}^n \xi_j \lambda_j^{-\beta} e_j(s)$,
 converges in $L_2(\mathbb{P})$ uniformly in $s$:
	$\lim_{n\to\infty} \sup_{s\in \Gamma} \mathbb{E}\left(\lvert u(s) - u_n(s)\rvert^2 \right) =0.$
	Moreover,  $u_n$ converges to $u$ in $\mathcal{L}_2(\Omega, L_2(\Gamma))$ (i.e., 
	$\lim_{n\to\infty}  \left\| u - u_n\right\|_{\mathcal{L}_2(\Omega,L_2(\Gamma))} = 0)$,
	and the rate of convergence of this approximation is 
	$\|u-u_n\|_{\mathcal{L}_2(\Omega, L_2(\Gamma))} \leq C n^{-(\alpha-\nicefrac12)}$,
	where $C>0$ is a constant that does not depend on $n$.
\end{Proposition}

\section{Proofs of the results in Section~\ref{sec:regularity}}\label{sec:proof}

We begin by obtaining a characterization of the $\dot{H}^2_{L_0}(\Gamma)$, which is needed in the proofs. We state and prove the result in the more general framework of generalized Kirchhoff vertex conditions so that we can also use it in Section~\ref{sec:numerical}.

\begin{Proposition}\label{prp:Hdot2}
	Let $\alpha\in\mathbb{R}$ and $L_{\alpha}$ be the operator \eqref{eq:operatorL} endowed with the generalized Kirchhoff vertex conditions \eqref{eq:kirchhoff}.
	Further, either let Assumption \ref{assum1_alt} hold, or let Assumption~\ref{assum1} hold with 
	$S = \frac{\lvert \alpha\rvert}{d_0}$ (see Remark~\ref{rmk1}). 
	Then, 
	\begin{equation}\label{eq:Hdot2genKirk}
		\dot{H}^2_{L_\alpha}(\Gamma) \cong 	\left\{u\in\widetilde{H}^2(\Gamma): u \in C(\Gamma),\,  \mbox{$\forall v\in\mathcal{V},\,\, \H(v)\sum_{e\in\mathcal{E}} \partial_e u(v) = \alpha u(v)$} \right\}.
	\end{equation}
	In particular, when $\alpha=0$, we have $S=0$, so the remaining conditions
	are automatically satisfied, and we have the following identification:
	\begin{equation}\label{eq:Hdot2}
		\dot{H}^2_{L_0}(\Gamma) \cong 	\left\{u\in\widetilde{H}^2(\Gamma): u \in C(\Gamma)  \mbox{ and $\forall v\in\mathcal{V},\,\, \sum_{e\in\mathcal{E}} \partial_e u(v) = 0$} \right\}.
	\end{equation}
\end{Proposition}

\begin{proof}
	We consider the following notation:
	\begin{equation}\label{eq:fracHkirc}
		H^2_\alpha(\Gamma) = \left\{u\in\widetilde{H}^2(\Gamma): u \in C(\Gamma) \mbox{ and $\forall v\in\mathcal{V},\,\H(v)\sum_e \partial_e u(v) = \alpha u(v)$} \right\}.
	\end{equation}
	The equality of the sets $\dot{H}^2_{L_\alpha}(\Gamma) = H^2_\alpha(\Gamma)$ follows directly from \eqref{domainL}, with the choices for $A_v$ and $B_v$ given by \eqref{eq:generalized_kirchoff}, and Theorem~\ref{compactLinverse}. Further, by Theorem~\ref{compactLinverse}, under these conditions, $L_\alpha$ is a self-adjoint and positive operator with a compact resolvent; thus, 
	$(\dot{H}^2_{L_\alpha}(\Gamma), \|\cdot\|_{\dot{H}^2_{L_\alpha}(\Gamma)})$ is
	well-defined. 
	It remains to prove the equivalence of the norms of $\dot{H}^2_{L_\alpha}(\Gamma)$ 
	and $\widetilde{H}^2(\Gamma)$ for functions in $\dot{H}^2_{L_\alpha}(\Gamma)$. 
	To this end, by the closed graph theorem, $(\dot{H}_{L_\alpha}^2(\Gamma), 
	\|\cdot\|_{\dot{H}_{L_\alpha}^2(\Gamma)}) \hookrightarrow (H^2_\alpha(\Gamma) , 
	\|\cdot\|_{\widetilde{H}^2_\alpha(\Gamma)}).$
	Indeed, $\lambda_{j,\alpha}\to \infty$ as $j\to\infty$, where $(\lambda_{j,\alpha})_{j\in\mathbb{N}}$ are the eigenvalues of $L_\alpha$, which 
	shows the continuous embedding $(\dot{H}_{L_\alpha}^2(\Gamma), \|\cdot\|_{\dot{H}_{L_\alpha}^2(\Gamma)}) 
	\hookrightarrow (L_2(\Gamma) , \|\cdot\|_{L_2(\Gamma)}).$ Further, let $I$ denote the 
	inclusion map $I:\dot{H}^2_{L_\alpha}(\Gamma) \to H^2_\alpha(\Gamma)$.  
	If $\phi_N \to 0$ in $\dot{H}^2_{L_\alpha}(\Gamma)$, then $\phi_N\to 0$ 
	in $L_2(\Gamma)$. In contrast, if $I(\phi_N) \to \phi$, then 
	$\|\phi_N - \phi\|_{L_2(\Gamma)}\leq \|\phi_N-\phi\|_{\widetilde{H}^2(\Gamma)} 
	\to 0$. Thus, $\phi = 0$ because $\phi_N \to 0$ in $L_2(\Gamma)$. By the closed 
	graph theorem, $I$ is a bounded operator.
	Conversely, by repeating the same argument from the previous inclusion, 
	$(H^2_\alpha(\Gamma), \|\cdot\|_{\widetilde{H}^2(\Gamma)}) 
	\hookrightarrow (\dot{H}_{L_\alpha}^2(\Gamma), \|\cdot\|_{\dot{H}_{L_\alpha}^2(\Gamma)})$ from the closed graph theorem, which proves the equivalence of norms.
\end{proof}

\begin{Proposition}\label{prp:Hdot1}
	For $\beta\in [0,1]$, we have the identification $\dot{H}^\beta_{L_0}(\Gamma) \cong H^\beta(\Gamma)$.
\end{Proposition}
\begin{proof}
	The case $\beta=0$ is clear because $\dot{H}^\beta_{L_0}(\Gamma) = L_2(\Gamma) =  H^\beta(\Gamma)$. We now consider the case $\beta=1$. By \eqref{eq:Hdot2}, $D(L_0) \subset H^1(\Gamma)$. Furthermore, for $u\in D(L_0)$, we can apply \eqref{positivedefinite} to
	define the energetic norm 
	$$\|u\|_E^2 = (L_0u, u) = (\kappa^2 u, u) + \sum_{e\in\mathcal{E}} \int_e \H(x) (u_e'(x))^2 dx.$$
	Next, $\H$ is continuous, and $\Gamma$ is compact, so $\H$ is bounded. Further, as we also have $\kappa\geq \kappa_0>0, \H\geq \H_0>0$, $\kappa\in L^\infty(\Gamma)$, it is immediate that $\|\cdot \|_E$ is equivalent to $\|\cdot\|_{H^1(\Gamma)}$. Further, $L_0$ is coercive and, by Theorem~\ref{compactLinverse}, self-adjoint. Therefore, the remaining proof for the case $\beta=1$ follows from the proof of \cite[Proposition~3]{BSW2022}. Finally, by real interpolation \cite[Lemma 5]{BSW2022} and the results for the cases $\beta=0$ and $\beta=1$, we find that 
	$\dot{H}^\beta_{L_0 }(\Gamma) \cong H^\beta(\Gamma)$ for $\beta=(0,1)$.
\end{proof}

\begin{Corollary}\label{cor:sobembedHdot}
	Let $\beta>\nicefrac12$ and define $\widetilde{\beta} = \beta - \nicefrac12$ if $\beta\leq 1$ and $\widetilde{\beta}=\nicefrac12$ if $\beta>1$. Then we have the embedding $\dot{H}^\beta_{L_0}(\Gamma) \hookrightarrow C^{0,\widetilde{\beta}}(\Gamma)$.
\end{Corollary}
\begin{proof}
	The proof follows directly from Proposition~\ref{prp:Hdot1} and \cite[Theorem 2]{BSW2022}.
\end{proof}

We are now in a position to prove Lemma~\ref{lem:regularity2} and Corollary~\ref{cor:covfunccont} from Section~\ref{sec:regularity}.
\begin{proof}[Proof of Lemma \ref{lem:regularity2}]
	For $\beta > \nicefrac14$, by Corollary~\ref{cor:sobembedHdot}, 
	$$
	L_0^{-\beta}:L_2(\Gamma)\to \dot{H}_{L_0}^{2\beta}(\Gamma) \cong H^{2\beta}(\Gamma) \hookrightarrow C^{0,2\widetilde{\beta}}(\Gamma).
	$$
	This result shows that ${L_0^{-\beta}:L_2(\Gamma)\to C^{0,2\widetilde{\beta}}(\Gamma)}$ is a bounded linear operator. From this point on, the remaining proof is identical to those in \cite[Lemma~1]{BSW2022} and \cite[Lemma~2]{BSW2022}. Indeed, the proofs do not use the specific form of the differential operator $L_0$; they only use that ${L_0^{-\beta}:L_2(\Gamma)\to C^{0,2\widetilde{\beta}}(\Gamma)}$ is a bounded linear operator.
\end{proof}

\begin{proof}[Proof of Corollary \ref{cor:covfunccont}]
	By Lemma \ref{lem:regularity2} and the fact that $u$ is a modification of $u_0$, 
	$$
	\mathbb{E}(\lvert u(x) - u(y)\rvert^2) = \mathbb{E}(\lvert u_0(x)-u_0(y)\rvert^2) \leq \|L_0^{-\beta} \|_{L(L_2(\Gamma),C^{0,2\widetilde{\beta}}(\Gamma))}^2 d(x,y)^{2\widetilde{\alpha}},
	$$
	which implies that $u$ is continuous in $L_2$ at every point $x\in\Gamma$. Therefore, $\varrho$ is continuous at $\Gamma\times\Gamma$. The remaining claims follow from the proofs of \cite[Lemma 3 and Proposition 7]{BSW2022}.
\end{proof}

For the proof of Theorem~\ref{thm:characterization}, we must also obtain a characterization of the spaces $\dot{H}^\beta_{L_0}(\Gamma)$ for $0 < \beta < 2$, $\beta\neq 1/2$ and $\beta\neq 3/2$. 
We first obtain the identification of $\dot{H}^\beta_{L_0}(\Gamma)$ 
for $1<\beta<2$, with $\beta\neq \nicefrac32$. 
We begin by stating and proving a few auxiliary results that are useful in the proof of the
characterization. 

\begin{Lemma}\label{lem:boundedderiv}
	Let $\Gamma$ be a compact metric graph and fix  
	$1<\beta<2$. Then, the restriction of $\nabla$ to $H^\beta(\Gamma)$
	maps to $H^{\beta-1}(\Gamma)$ and is a bounded operator. That is,
	for every $1<\beta<2$,
	$\nabla: H^\beta(\Gamma) \to H^{\beta-1}(\Gamma)$ is a bounded linear operator.
\end{Lemma}

\begin{proof}
	Fix any $1<\beta<2$ and let $u\in H^\beta(\Gamma)$. Additionally, take any 
	$u_0\in H^1(\Gamma)$ and $u_1\in \widetilde{H}^2_C(\Gamma)$ such that
	$u = u_0+u_1$. Then,
	\begin{align*}
		K(t,\nabla u; L^2(\Gamma), H^1(\Gamma))^2&\leq 
		\|\nabla u_0\|_{L^2(\Gamma)}^2 + t^2\|\nabla u_1\|_{H^1(\Gamma)}^2 \\
		&\leq \|u_0\|_{H^1(\Gamma)}^2 + t^2 \|u_1\|_{\widetilde{H}^2(\Gamma)}, 
	\end{align*}
	because $\nabla u = \nabla u_0 + \nabla u_1$, 
	where $K(\cdot,\cdot;\cdot,\cdot)$ is the $K$-functional from the $K$-method of real interpolation, e.g. \cite[Equation (28)]{BSW2022}. Hence, as ${u_0\in H^1(\Gamma)}$ and 
	${u_1\in \widetilde{H}^2_C(\Gamma)}$ were arbirary,
	$K(t, \nabla u; L^2(\Gamma), H^1(\Gamma)) \leq K(t,u;H^1(\Gamma), \widetilde{H}_C^2(\Gamma)).$
	Therefore, 
	$\|\nabla u\|_{H^{\beta-1}(\Gamma)} \leq 
	\|u\|_{(H^1(\Gamma),\widetilde{H}^2_C(\Gamma))_{\beta-1}} = \|u\|_{H^\beta(\Gamma)}$.
\end{proof}

The following Lemma is useful throughout the paper:
\begin{Lemma}\label{lem:restredgebounded}
	For any $0\leq \beta \leq 2$ and any edge $e\in\mathcal{E}$, the restriction operators
	${R_e: H^\beta(\Gamma) \to H^\beta(e)}$ and $R_e: \widetilde{H}^\beta(\Gamma) \to H^\beta(e),$ given by $R_e(f) = f\vert_e$, are bounded linear operators.
\end{Lemma}
\begin{proof}
	The restrictions $R_e:\widetilde{H}^2(\Gamma) \to H^2(e)$ and
	$R_e:L_2(\Gamma)\to L_2(e)$ are bounded linear operators. 
	Thus, the restriction to an edge $e\in\mathcal{E}$ is a couple map (see, e.g., \cite[Appendix A]{BSW2022} or \cite{chandlerwildeetal}) from 
	$(L_2(\Gamma), \widetilde{H}^2(\Gamma))$ to $(L_2(e), H^2(e))$. Therefore,
	from \cite[Theorem 2.2]{chandlerwildeetal}, for every $0<s<2$, the restriction map
	$R_e: \widetilde{H}^\beta(\Gamma)\to H^\beta(e)$ is a bounded linear operator.
	By an analogous argument, $R_e: H^\beta(\Gamma)\to H^\beta(e)$ is a bounded linear operator.
\end{proof}

In the following trace theorem, we introduce the trace operator on $H^s(\Gamma)$. 

\begin{Theorem}\label{thm:traceTheorem}
	Let $\gamma: C(\Gamma)\cap \bigoplus_{e\in\mathcal{E}} C^2(e) \to \mathbb{R}^{\sum_{v\in\mathcal{V}} d_v}$ be the trace operator defined as ${\gamma u = \{u\lvert_e(0), u\rvert_e(l_e)\}_{e\in\mathcal{E}}}$. That is, the trace operator contains the values of the function $u$ at the vertices of each edge. Then, for every $\nicefrac12 < s \leq 2$, $\gamma$ has a unique extension to a bounded linear operator 
	$
	\gamma : H^{s}(\Gamma) \to\mathbb{R}^{\sum_{v\in\mathcal{V}} d_v}.
	$
	Similarly, for every $\nicefrac12 < s \leq 2$, $\gamma$ also has a unique extension to a bounded linear operator
	$\gamma : \widetilde{H}^{s}(\Gamma) \to\mathbb{R}^{\sum_{v\in\mathcal{V}} d_v}.$ 
\end{Theorem}

\begin{proof}
	Fix any edge $e\in\mathcal{E}$ and observe that, by Theorem 3.37 of \cite{mclean}, some constant $C_e>0$ exists such that
	$\big\lvert u\lvert_e(0)\big\rvert + \big\lvert u\lvert_e(l_e)\big\rvert \leq C \|u_e\|_{H^{{s}}(e)}$.
	Next, by Lemma~\ref{lem:restredgebounded}, a constant $\widetilde{C}_e>0$ exists such that
	$\|u_e\|_{H^s(e)} \leq\widetilde{C}_e \|u\|_{H^{{s}}(\Gamma)}$.
	Therefore, 
	$$
	\sum_{e\in \mathcal{E}} \big\lvert u\lvert_e(0)\big\rvert + \big\lvert u\lvert_e(l_e)\big\rvert \lesssim   \|u\|_{H^s(\Gamma)}.
	$$
	Since $C(\Gamma)\cap \bigoplus_{e\in\mathcal{E}} C^2(e)$ is dense in $H^s(\Gamma)$, existence and uniqueness of the continuous extension follows.
	The proof of the second statement is analogous.
\end{proof}

In the next lemma, we collect some well-known results for $H^\beta(I)$, where $I\subset\mathbb{R}$
is an interval; thus, it can be applied to $H^\beta(e)$ for each edge $e\in\mathcal{E}$. 

\begin{Lemma}\label{lemma:wellknown}
	Let $e\in\mathcal{E}$ be an edge. 
	Then, for $0<s<\nicefrac12$, $H_0^s(e) = H^s(e)$. Further, for $u\in H_0^s(e)$, $ 0<s<\infty$, and $w\in L_2(e)$, $(u,w) = \langle u, w \rangle_{H_0^s(e)\times H^{-s}(e)}$.
\end{Lemma}
For a proof, see \cite[Theorem~3.40(i)]{mclean} and \cite{yagiHinfinityfunctional}. 
The following weak version of integration by parts is needed to prove the characterization of the fractional spaces:

\begin{Lemma}\label{lem:weakintbyparts}
	Let $\nicefrac12 < \beta \leq 1$, $f\in \widetilde{H}^\beta(\Gamma)$ and $g\in \widetilde{H}^2(\Gamma)$. 
	Then,
	$$
	\sum_{e\in\mathcal{E}} \int_e f_e(x)\Delta g_e(x)\,dx = 
	\sum_{v\in\mathcal{V}}\sum_{e\in\mathcal{E}_v} f\vert_e(v) \partial_e g(v) - 
	\sum_{e\in\mathcal{E}} \langle g_e', f_e' \rangle_{H^{1-\beta}(e)\times H^{\beta-1}(e)}.
	$$
	In particular, if $f\in \widetilde{H}^\beta(\Gamma)$ and $g\in \widetilde{H}^2(\Gamma)$,
	where $g$ satisfies
	the Kirchhoff vertex conditions and $f$ is continuous at each vertex
	$v\in\mathcal{V}$, we have
	\begin{equation}\label{eq:genintbypartsfrac_case2}
		\sum_{e\in\mathcal{E}} \int_e f_e(x)\Delta g_e(x)\,dx = - 
		\sum_{e\in\mathcal{E}} \langle g_e', f_e' \rangle_{H^{1-\beta}(e)\times H^{\beta-1}(e)}.
	\end{equation}
\end{Lemma}

\begin{proof}
	Initially, we let $f\in \widetilde{H}^1(\Gamma)$. Then, $0 \leq 1-\beta<\nicefrac12$, so we can use Lemma~\ref{lemma:wellknown} together with integration by parts to obtain the following:
	\begin{align}
		\sum_{e\in\mathcal{E}} \int_e f_e(x)\Delta g_e(x)\,dx &= \sum_{v\in\mathcal{V}}
		\sum_{e\in \mathcal{E}_v} f\vert_e(v) \partial_e g(v) 
		- \sum_{e\in\mathcal{E}}(f_e', g_e')_{L^2(e)}\nonumber\\
		&= \sum_{v\in\mathcal{V}}
		\sum_{e\in \mathcal{E}_v} f\vert_e(v) \partial_e g(v) 
		- \sum_{e\in\mathcal{E}}
		\langle g_e', f_e'\rangle_{H^{1-\beta}(e)\times H^{\beta-1}(e)}.\label{eq:intbypartsid}
	\end{align}
	By Theorem \ref{thm:traceTheorem}, $\gamma:\widetilde{H}^{\beta}(\Gamma)\to \mathbb{R}^{\sum_{v\in\mathcal{V}}d_v}$ is a bounded operator. Because $\beta\leq 1$, we have $f,g\in \widetilde{H}^{\beta}(\Gamma)$. Thus, for each $e\in \mathcal{E}, v\in\mathcal{E}_v$,
	\begin{equation}\label{eq:traceineqintb}
		\lvert f\vert_e (v) \partial_e g(v)\rvert \leq 
		\|\gamma(f)\|_{\mathbb{R}^{\sum_{v\in\mathcal{V}}}} \|\gamma(\nabla g)\|_{\mathbb{R}^{\sum_{v\in\mathcal{V}}d_v}}
		\lesssim 
		\|f\|_{\widetilde{H}^{\beta}(\Gamma)} 
		\| \nabla g\|_{\widetilde{H}^{\beta}(\Gamma)}.
	\end{equation}
	Further, for each $e\in\mathcal{E}$, $D:H^\beta(e)\to H^{\beta-1}(e)$ 
	is a bounded operator, and thus,
	\begin{align}
		\lvert\langle g_e',f_e'\rangle_{H^{1-\beta}(e)\times H^{\beta-1}(e)}\rvert &\leq 
		\|g_e'\|_{H^{1-\beta}(e)}\|f_e'\|_{H^{\beta-1}(e)}\nonumber
		\lesssim 
		\|f_e\|_{H^\beta(e)}\|g_e'\|_{H^{1-\beta}(e
			)}\nonumber\\
		&\lesssim \|f\|_{\widetilde{H}^\beta(\Gamma)}
		\|\nabla g\|_{\widetilde{H}^{1-\beta}(\Gamma)},\label{eq:ineqlaplfrac}
	\end{align}
	since, by Lemma~\ref{lem:restredgebounded}, the restriction to an edge $e$ from $\widetilde{H}^\beta(\Gamma)$ to $H^\beta(e)$ is 
	a bounded operator. As $\widetilde{H}^1(\Gamma)$ is dense in
	$\widetilde{H}^\beta(\Gamma)$,  \eqref{eq:traceineqintb} and \eqref{eq:ineqlaplfrac}
	imply that \eqref{eq:intbypartsid} can be continuously extended to hold for
	any functions $f\in \widetilde{H}^\beta(\Gamma)$ and ${g\in 
		\widetilde{H}^2(\Gamma)}$. 
	Therefore, for any $f\in \widetilde{H}^\beta(\Gamma)$ and any $g\in \widetilde{H}^2(\Gamma)$,
	we have
	$$\sum_{e\in\mathcal{E}} \int_e f_e(x) g_e'(x) \,dx = \sum_{v\in\mathcal{V}}
	\sum_{e\in \mathcal{E}_v} f\vert_e(v) \partial_e g(v) 
	- \sum_{e\in\mathcal{E}}
	\langle g, f_e'\rangle_{H^{1-\beta}(e)\times H^{\beta-1}(e)}.$$
\end{proof}

\begin{Lemma}\label{lem:weakintbypartsLaplacian}
	Let $\nicefrac32 < \beta < 2$, $f\in \widetilde{H}^\beta(\Gamma)$,
	and $g\in \widetilde{H}^1(\Gamma)$. Then,
	$$\sum_{e\in\mathcal{E}} \int_e f_e'(x) g_e'(x)dx = \sum_{v\in\mathcal{V}} 
	\sum_{e\in\mathcal{E}_v} \partial_e f(v)g\vert_e(v) - 
	\sum_{e\in\mathcal{E}} \langle g_e, \Delta f_e \rangle_{H^{2-\beta}(e)\times H^{\beta-2}(e)}.$$ 
	In particular, if $f\in \widetilde{H}^\beta(\Gamma)$ satisfies the Kirchhoff vertex
	conditions and $g\in \widetilde{H}^1(\Gamma)$ is continuous at each
	vertex $v\in\mathcal{V}$,
	we have
	\begin{equation}\label{eq:genintbypartsfrac}
		\sum_{e\in\mathcal{E}} \int_e f_e'(x) g_e'(x) dx = - \sum_{e\in\mathcal{E}}
		\langle g,\Delta f\rangle_{H^{2-\beta}(e)\times H^{\beta-2}(e)}.
	\end{equation}
\end{Lemma}

\begin{proof}
	The proof is similar to the proof of Lemma~\ref{lem:weakintbyparts},
	so we provide the main steps, as the details are analogous.
	Let $f\in \widetilde{H}^2(\Gamma)$ and
	$g\in \widetilde{H}^1(\Gamma)$. 
	Observe that ${0 < 2-\beta<\nicefrac12}$, so we can obtain
	\begin{equation}
		\sum_{e\in\mathcal{E}} \int_e f_e'(x) g_e'(x)\,dx = \sum_{v\in\mathcal{V}}
		\sum_{e\in \mathcal{E}_v} \partial_e f(v) g\vert_e(v) 
		- \sum_{e\in\mathcal{E}}
		\langle g,\Delta f\rangle_{H^{2-\beta}(e)\times H^{\beta-2}(e)}.\label{eq:intbypartsid2}
	\end{equation}
	Now, some constant $C>0$ exist such that, for every $e\in\mathcal{E}$ and $v\in\mathcal{E}_v$,
	\begin{equation}\label{eq:traceineqintb2}
		\lvert \partial_e f(v) g\vert_e (v)\rvert \leq 
		C\|\nabla f\|_{\widetilde{H}^{\beta-1}(\Gamma)} 
		\| g\|_{\widetilde{H}^{\beta-1}(\Gamma)}.
	\end{equation}
	Further, $\Delta = D^2:H^\beta(e)\to H^{\beta-2}(e)$ 
	is a bounded operator. Thus,
	\begin{align}
		\lvert\langle g,\Delta f\rangle_{H^{2-\beta}(e)\times H^{\beta-2}(e)}\rvert &\leq 
		\|g\|_{H^{2-\beta}(e)}\|\Delta f\|_{H^{\beta-2}(e)}\nonumber\\
		&\leq 
		\widetilde{C}\|f\|_{H^\beta(e)}\|g\|_{H^{2-\beta}(e
			)}\leq \widehat{C}\|f\|_{\widetilde{H}^\beta(\Gamma)}\|g\|_{\widetilde{H}^{2-\beta}(\Gamma)}.\label{eq:ineqlaplfrac2}
	\end{align}
	Further, $\widetilde{H}^2(\Gamma)$ is dense in
	$\widetilde{H}^\beta(\Gamma)$; thus, \eqref{eq:traceineqintb2} and \eqref{eq:ineqlaplfrac2} imply that \eqref{eq:intbypartsid2} can be continuously extended to
	hold for any functions $f\in \widetilde{H}^\beta(\Gamma)$ and $g\in \widetilde{H}^1(\Gamma)$. 
	Therefore, for any $f\in \widetilde{H}^\beta(\Gamma)$ and $g\in \widetilde{H}^1(\Gamma)$,
	we have
	$$\sum_{e\in\mathcal{E}} \int_e f_e'(x) g_e'(x)\,dx = \sum_{v\in\mathcal{V}}
	\sum_{e\in \mathcal{E}_v}\partial_e f(v) g\vert_e(v) 
	- \sum_{e\in\mathcal{E}}
	\langle g_e,\Delta f_e\rangle_{H^{2-\beta}(e)\times H^{\beta-2}(e)}.$$
\end{proof}

Next, we derive two embeddings needed in the proof of the characterization.

\begin{Lemma}\label{lem:simpleembeddings}
	If $0\leq\beta\leq 2$, then 
	$\dot{H}^\beta_{L_0}(\Gamma) \hookrightarrow H^\beta(\Gamma) \hookrightarrow \widetilde{H}^\beta(\Gamma).$
\end{Lemma}
\begin{proof}
	First, if $0\leq \beta \leq 1$, then by Proposition~\ref{prp:Hdot1}, 
	$\dot{H}^\beta_{L_0}(\Gamma)\cong H^\beta(\Gamma)$.
	In particular, $\dot{H}^\beta_{L_0}(\Gamma) \hookrightarrow H^\beta(\Gamma)$.
	Furthermore, $H^1(\Gamma) \hookrightarrow \widetilde{H}^1(\Gamma)$; therefore, 
	by \cite[Theorem~2.2, item (vi)]{chandlerwildeetal} it follows that
	$H^\beta(\Gamma) \hookrightarrow \widetilde{H}^\beta(\Gamma)$. 
	Next, let $1\leq\beta\leq 2$.  By Proposition~\ref{prp:Hdot1}, 
	$\dot{H}^1_{L_0}(\Gamma) \cong H^1(\Gamma)$, and by Proposition~\ref{prp:Hdot2},
	$\dot{H}^2_{L_0}(\Gamma) \cong H^2_{K}(\Gamma)$, where
	$H^2_K(\Gamma)$ is the set the right-hand side of \eqref{eq:Hdot2}. By \cite[Lemma 5]{BSW2022}, 
	$\dot{H}^\beta_{L_0}(\Gamma)\cong (H^1(\Gamma),H^2_K(\Gamma))_{\beta-1}$.
	We observe that $H^2_K(\Gamma) \hookrightarrow \widetilde{H}^2_C(\Gamma)$. Therefore,
	by \cite[Theorem 2.2, item (vi)]{chandlerwildeetal}, $\dot{H}^\beta_{L_0}(\Gamma) \hookrightarrow H^\beta(\Gamma)$.
	Finally, as $H^1(\Gamma)\hookrightarrow \widetilde{H}^1(\Gamma)$ and
	$H^2(\Gamma)\hookrightarrow \widetilde{H}^2(\Gamma)$, it again follows that,
	by \cite[Theorem 2.2, item (vi)]{chandlerwildeetal}, $H^\beta(\Gamma)\hookrightarrow \widetilde{H}^\beta(\Gamma)$.
\end{proof}

\begin{Corollary}\label{cor:Lboundedfrac}
	For $\beta \in (0,2)$, 
	the operators $L_0^{\nicefrac{\beta}{2}-1}:L_2(\Gamma)\to H^{2-\beta}(\Gamma)$ and
	$L_0^{\nicefrac{\beta}{2}-1}:L_2(\Gamma)\to \widetilde{H}^{2-\beta}(\Gamma)$
	are bounded and linear.
\end{Corollary}
\begin{proof}
	We directly have that
	$L_0^{\nicefrac{\beta}{2}-1}:L_2(\Gamma)\to \dot{H}_{L_0}^{2-\beta}(\Gamma)$ is a bounded
	linear operator for any $0 < \beta < 2$. By Lemma~\ref{lem:simpleembeddings},
	the two inclusions $i_1:\dot{H}^{2-\beta}_{L_0} \to H^{2-\beta}(\Gamma)$
	and ${i_2:\dot{H}^{2-\beta}_{L_0}(\Gamma) \to \widetilde{H}^{2-\beta}(\Gamma)}$,
	given by $i_1(x) = x$ and $i_2(x)=x$, are continuous. Therefore, the two operators 
	$i_1\circ L_0^{\nicefrac{\beta}{2}-1} = L_0^{\nicefrac{\beta}{2}-1}:L_2(\Gamma)\to H^{2-\beta}(\Gamma)$
	and $i_2\circ L_0^{\nicefrac{\beta}{2}-1} = L_0^{\nicefrac{\beta}{2}-1}:L_2(\Gamma)\to \widetilde{H}^{2-\beta}(\Gamma)$ are bounded and linear.
\end{proof}

We can now prove the characterization of $\dot{H}^\beta_{L_0}(\Gamma)$, for 
$1 < \beta <2$, $\beta\neq \nicefrac32$.

\begin{Theorem}\label{thm:charHdotFrac2}
	If $1 < \beta < \nicefrac32$, then 
	$\dot{H}^\beta_{L_0}(\Gamma) \cong H^{\beta}(\Gamma) \cong \widetilde{H}^\beta(\Gamma)\cap C(\Gamma),$
	where $H^{\beta}(\Gamma) = (H^1(\Gamma), \widetilde{H}_C^2(\Gamma))_{\beta-1}$.
	For $\nicefrac32 < \beta \leq 2$, we have that
	\begin{align*}
		\dot{H}^\beta_{L_0}(\Gamma) &\cong H^{\beta}(\Gamma)\cap \left\{u\in H^\beta(\Gamma): \forall v\in\mathcal{V},\,\sum_{e\in\mathcal{E}_v} \partial_e u(v) = 0\right\}\\
		&\cong \widetilde{H}^{\beta}(\Gamma)\cap \left\{u\in \widetilde{H}^\beta(\Gamma): 
		\hbox{$u$ is continuous and } \forall v\in\mathcal{V},\,\sum_{e\in\mathcal{E}_v} \partial_e u(v) = 0\right\}.
	\end{align*}
\end{Theorem}
\begin{proof}
	Let $H^2_K(\Gamma)$ be the set given by the right-hand side of \eqref{eq:Hdot2} and
	observe that, by Propositions~\ref{prp:Hdot1} and \ref{prp:Hdot2}, we have  $\dot{H}_{L_0}^1(\Gamma) \cong H^1(\Gamma)$ and
	$\dot{H}_{L_0}^2(\Gamma)\cong H^2_K(\Gamma)$. Similarly, let $\Delta$ be the Kirchhoff Laplacian and observe that Propositions~\ref{prp:Hdot1} and \ref{prp:Hdot2}, in particular, imply
	that $\dot{H}_{-\Delta + I}^1(\Gamma) \cong H^1(\Gamma)$ and $\dot{H}_{-\Delta + I}^2(\Gamma)\cong H^2_K(\Gamma)$.
	Therefore, by \cite[Lemma 5]{BSW2022}, 
	$\dot{H}^\beta_{L_0}(\Gamma) \cong (H^1(\Gamma), H^2_K(\Gamma))_{\beta-1} \cong \dot{H}_{-\Delta+I}^\beta(\Gamma)$. 
	Therefore, it is enough to show the result for the operator $L_0 = -\Delta + I$.
	
	By Lemma~\ref{lem:simpleembeddings}, we have the continuous embedding $\dot{H}^\beta_{-\Delta+I}(\Gamma) \hookrightarrow H^\beta(\Gamma)$, for $1<\beta<2$.
	In particular, a constant $C_\beta>0$ exists such for every $u\in \dot{H}^\beta_{-\Delta+I}(\Gamma)$,
	\begin{equation}\label{eq:continjecthdot}
		\|u\|_{H^{\beta}(\Gamma)} \leq C_\beta \|(-\Delta+I)^\beta u\|_{L_2(\Gamma)} = C_\beta \|u\|_{\dot{H}^\beta_{-\Delta+I}(\Gamma)}.
	\end{equation}
	
	Again, by Lemma~\ref{lem:simpleembeddings}, we have
	$H^\beta(\Gamma) \hookrightarrow \widetilde{H}^\beta(\Gamma)$ for $1<\beta<2$. Moreover, $\dot{H}^\beta_{-\Delta +I}(\Gamma) \subset C(\Gamma)$ (indeed, for instance, 
	$\dot{H}^\beta_{-\Delta +I}(\Gamma)\subset H^\beta(\Gamma)$ and $H^\beta(\Gamma)\subset C(\Gamma)$. Alternatively, one could use Corollary~\ref{cor:sobembedHdot}).
	Therefore, 
	$\dot{H}_{-\Delta+I}^\beta(\Gamma) \hookrightarrow \widetilde{H}^\beta(\Gamma)\cap C(\Gamma)$,
	for $1<\beta<2$.
	Let us now show that if $1<\beta<\nicefrac32$, we also have the converse continuous inclusion 
	${H^\beta(\Gamma)\hookrightarrow\dot{H}^\beta_{-\Delta+I}(\Gamma)}$, and that 
	 $\widetilde{H}^\beta(\Gamma)\cap C(\Gamma) \hookrightarrow 
	\dot{H}_{-\Delta+I}^\beta(\Gamma)$. 
	Here, we adapt
	the approach of \cite[Section 4]{yagiHinfinityfunctional}.
	To this end, let ${1<\beta < \nicefrac32}$, and write $L_0$ instead of $I-\Delta$ to keep the notation simple. Furthermore, the following inequality will also be used several times for this proof, and also in the remaining of this section. Let $0<s<\nicefrac{1}{2}$, $f\in H^s(\Gamma)$ and $g\in H^{-s}(\Gamma)$, then
	\begin{equation}\label{eq:ineq_fracSov_CharSob}
		\lvert(f, g)\rvert \leq 
		\sum_{e\in\mathcal{E}} \lvert (f, g)_{L_2(e)}\rvert = 
		\sum_{e\in\mathcal{E}} \lvert \langle f, g\rangle_{H^{s}(e)\times H^{-s}(e)}\rvert
		\leq \sum_{e\in\mathcal{E}} 
		\|f\|_{H^{s}(e)} \|g\|_{H^{-s}(e)}.
	\end{equation}
	
	Now, observe that $L_0^{\nicefrac{\beta}{2}-1}$ maps $\dot{H}^2_{L_0}(\Gamma)$ onto
	$\dot{H}^{4-\beta}_{L_0}(\Gamma) \subset \dot{H}^2_{L_0}(\Gamma)$.
	Thus, for any ${g\in \dot{H}^2_{L_0}(\Gamma)}$, we
	know that $L_0^{\nicefrac{\beta}{2}-1}g \in \dot{H}^2_{L_0}(\Gamma)$. In particular,
	from Proposition~\ref{prp:Hdot2} it follows that $L_0^{\nicefrac{\beta}{2}-1}g$
	satisfies the Kirchhoff vertex conditions. Thus, for $u\in H^\beta(\Gamma)$ and $g\in \dot{H}^2_{L_0}(\Gamma)$, we can apply integration by parts with the fact that $L_0^{\nicefrac{\beta}{2}-1}g$ satisfies the Kirchhoff's vertex conditions to obtain the following:
	\begin{align}
		(u, L_0^{\nicefrac{\beta}{2}} g) &= 
		(u, (I-\Delta)L_0^{\nicefrac{\beta}{2}-1}g)= (u,L_0^{\nicefrac{\beta}{2}-1}g) + (\nabla u, \nabla L_0^{\nicefrac{\beta}{2}-1}g).\label{eq:expression_char_Lbeta_1}
	\end{align}
	Let us now address the second term on the right-hand side of the above expression.
	By Corollary~\ref{cor:Lboundedfrac}, $L_0^{\nicefrac{\beta}{2}-1}$ is a bounded 
	operator from $L_2(\Gamma)$ to $H^{2-\beta}(\Gamma)$.
	By Lemma~\ref{lemma:wellknown}, Eq.~\eqref{eq:ineq_fracSov_CharSob}, 
	Lemma \ref{lem:restredgebounded}, and the fact that $0 < \beta - 1<\nicefrac12$,
	\begin{align*}
		\lvert(\nabla u, \nabla L_0^{\nicefrac{\beta}{2}-1} g)\rvert &\leq 
		 \sum_{e\in\mathcal{E}} 
		\|u_e'\|_{H^{\beta-1}(e)} \|D L_0^{\nicefrac{\beta}{2}-1} g\|_{H^{1-\beta}(e)} \\
		&\lesssim \sum_{e\in\mathcal{E}} \|u\|_{H^{\beta}(e)} \|L_0^{\nicefrac{\beta}{2}-1} g\|_{H^{2-\beta}(e)}\\
		&\lesssim \lvert\mathcal{E}\rvert 
		\|u\|_{H^{\beta}(\Gamma)} \|L_0^{\nicefrac{\beta}{2}-1} g\|_{H^{2-\beta}(\Gamma)} \lesssim \lvert\mathcal{E}\rvert \|u\|_{H^\beta(\Gamma)} \|g\|_{L_2(\Gamma)}.
	\end{align*}
	
	By a similar argument, and using that for every edge $e\in\mathcal{E}$, ${H^\beta(e)\hookrightarrow H^{\beta-1}(e)}$ and
	${H^{2-\beta}(e) \hookrightarrow H^{1-\beta}(e)}$, we obtain that, for $u \in H^{\beta}(\Gamma)$ and $g\in \dot{H}^2_{L_0}(\Gamma)$,
	\begin{align*}
		\lvert (u, L_0^{\nicefrac{\beta}{2}-1} g)\rvert  
		\lesssim \lvert\mathcal{E}\rvert  \|u\|_{H^\beta(\Gamma)} \|g\|_{L_2(\Gamma)},	
	\end{align*}
	From \eqref{eq:expression_char_Lbeta_1} and the above
	estimates, we have that
		$\lvert (u, L_0^{\nicefrac{\beta}{2}}g)\rvert \lesssim \lvert\mathcal{E}\rvert
		\|u\|_{H^\beta(\Gamma)}\|g\|_{L_2(\Gamma)}$.
	Since $g\in \dot{H}^2_{L_0}(\Gamma)$ is arbitrary and 
	$\dot{H}^2_{L_0}(\Gamma)$ is dense in $L_2(\Gamma)$,
	${u\in D(L_0^{\nicefrac{\beta}{2}}) = \dot{H}^\beta_{L_0}(\Gamma)}$ and 
	$
	\|u\|_{\dot{H}_{L_0}^\beta(\Gamma)} = \|L_0^{\nicefrac{\beta}{2}} u\|_{L_2(\Gamma)} \lesssim \lvert\mathcal{E}\rvert \|u\|_{H^\beta(\Gamma)}.
	$
	This proves the continuous inclusion $H^\beta(\Gamma)\hookrightarrow\dot{H}^\beta_{L_0}(\Gamma)$ for $1<\beta < \nicefrac32$. By repeating the same proof,
	but taking ${u\in \widetilde{H}^\beta(\Gamma)\cap C(\Gamma)}$ instead of
	$u\in H^\beta(\Gamma)$, and because, by Lemma~\ref{lem:restredgebounded}, the restriction map from $\widetilde{H}^\beta(\Gamma)$
	to $H^\beta(e)$ for any edge $e\in\mathcal{E}$ is a bounded operator, we obtain
	the continuous inclusion $\widetilde{H}^\beta(\Gamma)\cap C(\Gamma)\hookrightarrow
	\dot{H}^\beta_{L_0}(\Gamma)$ for $1<\beta < \nicefrac32$.
	This completes the characterization for $1<\beta<\nicefrac32.$
	
	To prove the characterization for $\nicefrac32 < \beta < 2$, we introduce 
	$$H^\beta_K(\Gamma) = H^{\beta}(\Gamma)\cap \left\{u\in \widetilde{H}^\beta(\Gamma): \forall v\in\mathcal{V},\,\sum_{e\in\mathcal{E}_v} \partial_e u(v) = 0\right\},$$
	endowed with the norm inherited from $H^\beta(\Gamma)$.
	We begin by proving the continuous inclusion
	$\dot{H}^\beta_{L_0}(\Gamma) \hookrightarrow H^\beta_K(\Gamma)$. 
	We recall that, in the first part of the proof, we found that 
	$\|u\|_{H^{\beta}(\Gamma)} \leq C_\beta \|u\|_{\dot{H}^\beta_{L_0}(\Gamma)}$ 
	and $\dot{H}_{L_0}^\beta(\Gamma) \subset H^\beta(\Gamma)$. We
	must validate that $\dot{H}^\beta_{L_0}(\Gamma) \subset H^\beta_K(\Gamma)$. Given the previous observation, we must demonstrate that the functions in 
	$\dot{H}^\beta_{L_0}(\Gamma)$
	are continuous and satisfy the Kirchhoff vertex conditions. 
	In the first part of the proof, we found that
	$\dot{H}^\beta_{L_0}(\Gamma) \cong (H^1(\Gamma), H^2_K(\Gamma))_{\beta-1}\subset H^1(\Gamma)$, so all functions in 
	$\dot{H}^\beta_{L_0}(\Gamma)$ are continuous.
	
	To demonstrate that any function in $\dot{H}_{L_0}^\beta(\Gamma)$ satisfies the 
	Kirchhoff vertex conditions, we let $u\in \dot{H}_{L_0}^\beta(\Gamma)$. 
	We observe that $\dot{H}^2_{L_0}(\Gamma)$ 
	is dense in $\dot{H}^\beta_{L_0}(\Gamma)$. Furthermore, by Proposition~\ref{prp:Hdot2},
	$\dot{H}^2_{L_0}(\Gamma) \cong H^2_K(\Gamma)$.
	Thus, a sequence $u_k\in H^2_K(\Gamma)$ exists such that
	$u_k\to u$ as $k\to \infty$ in $\dot{H}^\beta_{L_0}(\Gamma)$. Because 
	$u_k \in H^2_K(\Gamma)$, it satisfies the Kirchhoff vertex conditions. Therefore, in particular,
	for each $k\in\mathbb{N}$ and each $v\in\mathcal{V}$, 
	$\sum_{e\in \mathcal{E}_v} \partial_e u_k(v) = 0$. 
Further, because $\nicefrac32 < \beta < 2$, Lemma~\ref{lem:boundedderiv}
	shows that $\nabla u \in H^{\beta-1}(\Gamma)$. Furthermore,
	we note that $\nicefrac12 < \beta-1$.
	Therefore, we can apply Theorem~\ref{thm:traceTheorem} to $\nabla u$
	and find that 
	\begin{align*}
		\lvert \sum_{e\in\mathcal{E}_v} \partial_e u(v) \rvert &= \lvert \sum_{e\in\mathcal{E}_v} \partial_e u(v) - 
		\partial_e u_k(v) \rvert\leq \sum_{e\in\mathcal{E}} 
		\big(\lvert\nabla (u - u_k) \lvert_e(0)\rvert + \lvert\nabla (u-u_k)\lvert_e(l_e)\rvert\big)\\
		&\leq \|\gamma(\nabla u - \nabla u_k) \|_{\mathbb{R}^{\sum_{v\in\mathcal{V}}d_v}} 
		\lesssim \|\nabla (u - u_k)\|_{H^{\beta-1}(\Gamma)}\\
		&\lesssim  \|u - u_k\|_{H^\beta(\Gamma)}
		\lesssim  \|u-u_k\|_{\dot{H}_{-\Delta+I}^\beta(\Gamma)} \rightarrow 0,\quad \hbox{as $k\to\infty$,}
	\end{align*}
	because $u_k \to u$ in $\dot{H}^\beta_{L_0}(\Gamma)$, where we apply the
	fact that $\nabla$ is a bounded operator from $H^\beta(\Gamma)$ to
	$H^{\beta-1}(\Gamma)$ and inequality \eqref{eq:continjecthdot}. 
	Thus,  $\dot{H}^\beta_{L_0}(\Gamma) \hookrightarrow H^\beta_K(\Gamma)$.
	
	It remains to 
	prove the converse continuous inclusion $H^\beta_K(\Gamma)\hookrightarrow \dot{H}^\beta_{L_0}(\Gamma)$, for $\nicefrac32<\beta<2$.  Let $u\in H^\beta_K(\Gamma)$ and $g\in \dot{H}^2_{L_0}(\Gamma)$ and
	observe that \eqref{eq:expression_char_Lbeta_1} applies to this case and
	that ${L_0^{\nicefrac{\beta}{2}-1}g \in 
		\dot{H}_{L_0}^{4-\beta}(\Gamma)\subset \dot{H}^2_{L_0}(\Gamma)}$.
	Therefore, $L_0^{\nicefrac{\beta}{2}-1}g$ is continuous,
	and because $u$ satisfies the Kirchhoff vertex conditions, 
	we can apply Lemma~\ref{lem:weakintbypartsLaplacian}.
	More precisely, we can use \eqref{eq:genintbypartsfrac}.
	We combine this with Lemma~\ref{lemma:wellknown}, 
	Lemma~\ref{lem:restredgebounded}, the fact that for any $e\in\mathcal{E}$, 
	$\Delta:H^\beta(e)\to H^{\beta-2}(e)$ is a bounded operator, and that, by Corollary~\ref{cor:Lboundedfrac}, $L_0^{\nicefrac{\beta}{2}-1}:L_2(\Gamma)\to H^{2-\beta}(\Gamma)$ is bounded, to obtain
	\begin{align*}
		\lvert (\nabla u, \nabla &L_0^{\nicefrac{\beta}{2}-1}g)\rvert = \left\lvert 
		\sum_{e\in \mathcal{E}} \langle L_0^{\nicefrac{\beta}{2}-1} g,
		\Delta u\rangle_{H^{2-\beta}(e)\times H^{\beta-2}(e)}\right\rvert \\
		&\leq \sum_{e\in\mathcal{E}} \|L_0^{\nicefrac{\beta}{2}-1} g\|_{H^{2-\beta}(e)} \|\Delta u\|_{H^{\beta-2}(e)}
		\leq \sum_{e\in\mathcal{E}} \|L_0^{\nicefrac{\beta}{2}-1}g\|_{H^{2-\beta}(e)} \|u\|_{H^\beta(e)}\\ 
		&\lesssim \lvert\mathcal{E}\rvert \|L_0^{\nicefrac{\beta}{2}-1} g\|_{H^{2-\beta}(\Gamma)} \|u\|_{H^\beta(\Gamma)}
		\lesssim \lvert\mathcal{E}\rvert \|g\|_{L_2(\Gamma)} \|u\|_{H^\beta(\Gamma)}.
	\end{align*}
	
	To address the other term in \eqref{eq:expression_char_Lbeta_1}, we recall that for every $e\in\mathcal{E}$, we have $H^\beta(e)\hookrightarrow H^{\beta-2}(e)$.
	Moreover, $0 < 2-\beta < \nicefrac12$. 
	Thus, we proceed in a similar manner to obtain 
	$\lvert (u, L_0^{\nicefrac{\beta}{2}-1} g)\rvert 
	\lesssim \lvert\mathcal{E}\rvert  \|u\|_{H^\beta(\Gamma)} \|g\|_{L_2(\Gamma)}$
	for ${u \in H^{\beta}(\Gamma)}$ and ${g\in \dot{H}^2_{L_0}(\Gamma)}$.
	From \eqref{eq:expression_char_Lbeta_1} and the above
	estimates, for any $u\in H^\beta_K(\Gamma)$ and any $w\in 
	\dot{H}^2_{L_0}(\Gamma)$,
	$
	\lvert (u, L_0^{\nicefrac{\beta}{2}}w) \rvert \lesssim  \lvert\mathcal{E}\rvert\|u\|_{H^\beta(\Gamma)}\|w\|_{L_0^2(\Gamma)}.
	$
	Therefore, as $g\in\dot{H}^2_{L_0}(\Gamma)$ is arbitrary and 
	$\dot{H}^2_{L_0}(\Gamma)$ is dense in $L_0^2(\Gamma)$, 
	${u\in D(L_0^{\nicefrac{\beta}{2}}) = \dot{H}^\beta_{L_0}(\Gamma)}$ and
	$$
	{\|u\|_{\dot{H}^\beta_{L_0}(\Gamma)} = \|L_0^{\nicefrac{\beta}{2}}u\|_{L_2(\Gamma)} 
	\lesssim \lvert\mathcal{E}\rvert \|u\|_{H^\beta(\Gamma)}}.
$$
	That is,  $H^\beta_K(\Gamma) \hookrightarrow 
	\dot{H}^\beta_{L_0}(\Gamma)$ for $\nicefrac32<\beta<2$. 
	Now, let
	$$\widetilde{H}^\beta_K(\Gamma) = \widetilde{H}^{\beta}(\Gamma)\cap \left\{u\in \widetilde{H}^\beta(\Gamma): \hbox{$u$ is continuous and } \forall v\in\mathcal{V},\,\sum_{e\in\mathcal{E}_v} \partial_e u(v) = 0\right\}.$$
	Similarly to the case $1<\beta<\nicefrac32$, we can repeat the same proof, but take
	$u\in \widetilde{H}_K^\beta(\Gamma)$ and replace $H^\beta_K(\Gamma)$
	with $\widetilde{H}_K^\beta(\Gamma)$ to obtain the continuous inclusion
	${\widetilde{H}^\beta_K(\Gamma) \hookrightarrow 
		\dot{H}^\beta_{L_0}(\Gamma)}$
	for $\nicefrac32<\beta<2$. 
\end{proof}

\begin{Remark}
	The above strategy cannot be adapted for generalized Kirchhoff vertex conditions. Indeed, if we try to use \eqref{eq:traceineqintb} in the proof of Theorem~\ref{thm:charHdotFrac2} for the case $\nicefrac32<\beta<2$,  we would obtain
	$(\nabla u, \nabla L_0^{\nicefrac{\beta}{2}-1}g) \leq C\|u\|_{H^\beta(\Gamma)} \|g\|_{H^{2\beta-3}(\Gamma)}.$ Because $\nicefrac32 < \beta < 2$, we have $0 < 2\beta -3$; therefore, it would not be possible to obtain the  bound
	$
	\lvert (\nabla u, \nabla L_0^{\nicefrac{\beta}{2}-1}g) \rvert \leq C\|u\|_{H^\beta(\Gamma)} \|g\|_{L_2(\Gamma)}.
	$
	Hence, this term must be removed, and we can do so under the Kirchhoff vertex conditions because the term then vanishes since the functions in the corresponding domain satisfy the Kirchhoff vertex conditions. The term also vanishes if one considers Dirichlet or Neumann vertex conditions (which completely decouple the metric graph), but for generalized Kirchhoff vertex conditions,
	the term does not vanish.
\end{Remark}

In the following result, we address the remaining characterizations
when ${\beta \in [0,1]}$ (recall that $\dot{H}^\beta_{L_0}(\Gamma) \cong H^\beta(\Gamma)$ was shown in Proposition~\ref{prp:Hdot1} for this case).

\begin{Theorem}\label{thm:charbeta01remaining}
	If $0\leq\beta\leq 1$, where $\beta\neq \nicefrac12$. Then, we have the characterization
	\begin{equation*}
		\dot{H}^\beta_{L_0}(\Gamma) \cong
		\begin{cases}
			\widetilde{H}^\beta(\Gamma) &
			\text{for $0\leq\beta<\nicefrac12$} \\
			\widetilde{H}^\beta(\Gamma)\cap C(\Gamma) & 
			\text{for $\nicefrac12<\beta\leq 1$.}
		\end{cases}
	\end{equation*}
\end{Theorem}
\begin{proof}
	First,  by Lemma~\ref{lem:simpleembeddings},
	$\dot{H}_{L_0}^\beta(\Gamma)\hookrightarrow \widetilde{H}^\beta(\Gamma)$, for
	$0<\beta<1$. If $\nicefrac12 < \beta < 1$, by Corollary~\ref{cor:sobembedHdot}, it follows that $\dot{H}_{L_0}^\beta(\Gamma) \subset C(\Gamma)$. Therefore, 
	$\dot{H}_{L_0}^\beta(\Gamma) \hookrightarrow \widetilde{H}^\beta(\Gamma)\cap C(\Gamma)$, for $\nicefrac12 < \beta\leq 1$.
	We must prove the converse continuous inclusions, so we proceed as in Theorem~\ref{thm:charHdotFrac2}. By the same arguments as in 
	Theorem~\ref{thm:charHdotFrac2}, it is sufficient to derive the result for ${L_0=-\Delta + I}$. 
	We first address the case $0<\beta<\nicefrac12$, where we let $u\in \widetilde{H}^\beta(\Gamma)$ and $g\in\dot{H}^2_{L_0}(\Gamma)$. We find that
	\begin{equation}\label{eq:identityadjointLchar}
		(u, L_0^{\nicefrac{\beta}{2}} g) = (u, L_0 L_0^{\nicefrac{\beta}{2}-1}g) = (u, L_0^{\nicefrac{\beta}{2}-1}g) - (u, \Delta L_0^{\nicefrac{\beta}{2}-1}g).
	\end{equation}
	Then, we focus on the second term on the right-hand side of the above expression. For every $e\in\mathcal{E}$, $\Delta : H^{2-\beta}(e) \to H^{-\beta}(e)$ is a bounded
	operator and, by Corollary~\ref{cor:Lboundedfrac}, 
	$L_0^{\nicefrac{\beta}{2}-1}:L_2(\Gamma)\to \widetilde{H}^{2-\beta}(\Gamma)$
	is a bounded operator. Moreover, by Lemma~\ref{lem:restredgebounded}, the restriction
	operator to each edge $e\in\mathcal{E}$ is a bounded operator. Therefore,
	for any $u\in \widetilde{H}^\beta(\Gamma)$ and any $g\in\dot{H}^2_{L_0}(\Gamma)$,
	by Lemma~\ref{lemma:wellknown} and \eqref{eq:ineq_fracSov_CharSob}, 
	\begin{align*}
		\lvert (u, \Delta L_0^{\nicefrac{\beta}{2}-1} g)  \rvert 
		&\leq \sum_{e\in\mathcal{E}} \|u\|_{H^\beta(e)} \|\Delta L_0^{\nicefrac{\beta}{2}-1} g\|_{H^{-\beta}(e)}
		\lesssim  \lvert \mathcal{E}\rvert \|u\|_{H^\beta(e)} \|L_0^{\nicefrac{\beta}{2}-1}g\|_{H^{2-\beta}(e)}\\
		&\lesssim \lvert \mathcal{E}\rvert\|u\|_{\widetilde{H}^\beta(\Gamma)} 
		\|L_0^{\nicefrac{\beta}{2}-1}g\|_{\widetilde{H}^{2-\beta}(\Gamma)}
		\lesssim \lvert 
		\mathcal{E}\rvert\|u\|_{\widetilde{H}^\beta(\Gamma)} \|g\|_{L_2(\Gamma)}. 
	\end{align*}
	Now, note that for any edge $e\in\mathcal{E}$, $H^{2-\beta}(e) \hookrightarrow H^{-\beta}(e)$. 
	Thus, we obtain, similarly,
	\begin{align*}
		\lvert (u, L_0^{\nicefrac{\beta}{2}-1} g)\rvert 
		\lesssim \lvert\mathcal{E}\rvert  \|u\|_{\widetilde{H}^\beta(\Gamma)} \|g\|_{L_2(\Gamma)}.	
	\end{align*}
	From \eqref{eq:identityadjointLchar} and the above
	estimates, we find that
		$\lvert (u, L_0^{\nicefrac{\beta}{2}}g)\rvert \lesssim \lvert\mathcal{E}\rvert
		\|u\|_{\widetilde{H}^\beta(\Gamma)}\|g\|_{L_2(\Gamma)}$.
	Since $g\in \dot{H}^2_{L_0}(\Gamma)$ is arbitrary and 
	$\dot{H}^2_{L_0}(\Gamma)$ is dense in $L_2(\Gamma)$,
	$u\in D(L_0^{\nicefrac{\beta}{2}}) = \dot{H}^\beta_{L_0}(\Gamma)$ and 
	$
	\|u\|_{\dot{H}_{L_0}^\beta(\Gamma)} = \|L_0^{\nicefrac{\beta}{2}} u\|_{L_2(\Gamma)} \lesssim
	\lvert\mathcal{E}\rvert \|u\|_{\widetilde{H}^\beta(\Gamma)}.
	$
	This statement proves the continuous inclusion $\widetilde{H}^\beta(\Gamma)\hookrightarrow\dot{H}^\beta_{L_0}(\Gamma)$ for $0<\beta < \nicefrac12$.
	
	It remains to show that  $\widetilde{H}^\beta(\Gamma)\cap C(\Gamma) \hookrightarrow \dot{H}_{L_0}^\beta(\Gamma),$
	for $\nicefrac12 < \beta\leq 1$.
	First, we observe that \eqref{eq:identityadjointLchar} also holds for $\nicefrac12 <\beta <1$. Then, we estimate both terms on the
	right-hand side of \eqref{eq:identityadjointLchar}. To this end,
	take any $u\in \widetilde{H}^\beta(\Gamma)\cap C(\Gamma)$ and
	any $g\in\dot{H}^2_{L_0}(\Gamma)$. Because ${L_0^{\nicefrac{\beta}{2}-1}g \in 
		\dot{H}_{L_0}^{4-\beta}(\Gamma)\subset \dot{H}^2_{L_0}(\Gamma)}$, 
	$L_0^{\nicefrac{\beta}{2}-1}g$ satisfies the Kirchhoff vertex conditions,
	and because $u$ is continuous, we can apply \eqref{eq:genintbypartsfrac_case2}.
	Further, we recall that, for each $e\in\mathcal{E}$ and
	any $s\in\mathbb{R}$, $D:H^s(e)\to H^{s-1}(e)$ is a bounded operator, and by
	Lemma~\ref{lem:restredgebounded}, the restrictions to edges are
	bounded. Finally, by Corollary~\ref{cor:Lboundedfrac},
	$L_0^{\nicefrac{\beta}{2}-1}$ is a bounded operator from 
	$L_2(\Gamma)$ to $\widetilde{H}^{2-\beta}(\Gamma)$. Therefore, 
	\begin{align*}
		\lvert (u, \Delta &L_0^{\nicefrac{\beta}{2}-1}g)\rvert = \left\lvert \sum_{e\in\mathcal{E}} 
		\langle D L_0^{\nicefrac{\beta}{2}-1}g, u_e' \rangle_{H^{1-\beta}(e)\times H^{\beta-1}(e)} \right\rvert\\
		&\leq \sum_{e\in\mathcal{E}} \|D L_0^{\nicefrac{\beta}{2}-1} g\|_{H^{1-\beta}(e)} \|u_e'\|_{H^{\beta-1}(e)}
		\lesssim  \sum_{e\in\mathcal{E}} \|L_0^{\nicefrac{\beta}{2}-1}g\|_{H^{2-\beta}(e)} \|u\|_{H^\beta(e)}\\
		&\lesssim  \lvert \mathcal{E}\rvert \|L_0^{\nicefrac{\beta}{2}-1}g\|_{\widetilde{H}^{2-\beta}(\Gamma)} \|u\|_{\widetilde{H}^\beta(\Gamma)}
		\lesssim \lvert \mathcal{E}\rvert \|u\|_{\widetilde{H}^\beta(\Gamma)}\|g\|_{L_2(\Gamma)}. 
	\end{align*}
	Note that for any edge $e\in\mathcal{E}$, we have
	$H^{2-\beta}(e) \hookrightarrow H^{1-\beta}(e)$ and $H^{\beta}(e) \hookrightarrow H^{\beta-1}(e)$. 
	Thus, by similar calculations, 
		$\lvert (u, L_0^{\nicefrac{\beta}{2}-1} g)\rvert 
		\lesssim \lvert\mathcal{E}\rvert  \|u\|_{\widetilde{H}^\beta(\Gamma)} \|g\|_{L_2(\Gamma)}$.	
	From \eqref{eq:identityadjointLchar} and the above
	estimates,  
		$\lvert (u, L_0^{\nicefrac{\beta}{2}}g)\rvert \lesssim \lvert\mathcal{E}\rvert
		\|u\|_{\widetilde{H}^\beta(\Gamma)}\|g\|_{L_2(\Gamma)}$.
	Because $g\in \dot{H}^2_{L_0}(\Gamma)$ is arbitrary and 
	$\dot{H}^2_{L_0}(\Gamma)$ is dense in $L_2(\Gamma)$, we have 
	$u\in D(L_0^{\nicefrac{\beta}{2}}) = \dot{H}^\beta_{L_0}(\Gamma)$ and 
	$
	\|u\|_{\dot{H}_{L_0}^\beta(\Gamma)} = \|L_0^{\nicefrac{\beta}{2}} u\|_{L_2(\Gamma)} \leq 
	\lvert\mathcal{E}\rvert (\widetilde{C}+\widehat{C}) \|u\|_{\widetilde{H}^\beta(\Gamma)}.
	$
	This  proves the continuous inclusion 
	$\widetilde{H}^\beta(\Gamma)\cap C(\Gamma)\hookrightarrow\dot{H}^\beta_{L_0}(\Gamma)$ 
	for $\nicefrac12<\beta < 1$ and concludes the proof.
\end{proof}

Finally, we can summarize the above results into a proof of Theorem~\ref{thm:characterization}.

\begin{proof}[Proof of Theorem~\ref{thm:characterization}]
	By Proposition~\ref{prp:Hdot1} and Theorem~\ref{thm:charbeta01remaining}, for $0 < \beta < \nicefrac12$,
	$
	\dot{H}^\beta_{L_0}(\Gamma) \cong H^\beta(\Gamma) \cong \widetilde{H}^\beta(\Gamma)
	$
	and for $\nicefrac12 < \beta \leq 1$,
	$\dot{H}^\beta_{L_0}(\Gamma) \cong H^\beta(\Gamma) \cong \widetilde{H}^\beta(\Gamma)\cap C(\Gamma).$
	By Theorem~\ref{thm:charHdotFrac2}, for $1<\beta<\nicefrac32$,
	$\dot{H}^\beta_{L_0}(\Gamma) \cong H^\beta(\Gamma) \cong \widetilde{H}(\Gamma)\cap C(\Gamma),$
	and for $\nicefrac32<\beta\leq 2$, 
	$\dot{H}^\beta_{L_0}(\Gamma) \cong H^\beta(\Gamma)\cap K^\beta(\Gamma) \cong \widetilde{H}^\beta(\Gamma)\cap C(\Gamma)\cap K^\beta(\Gamma).$
	The identification $\dot{H}^{\nicefrac{1}{2}}(\Gamma) \cong H^{\nicefrac{1}{2}}(\Gamma)$ follows from Proposition \ref{prp:Hdot1}.
	Finally, we let $\beta>\nicefrac12$ and define $\widetilde{\beta} = \beta - \nicefrac12$
	if $\beta\leq 1$ and $\widetilde{\beta}=\nicefrac12$ if $\beta>1$. 
	By Corollary~\ref{cor:sobembedHdot}, $\dot{H}^\beta_{L_0}(\Gamma) \hookrightarrow C^{0,\widetilde{\beta}}(\Gamma)$, which concludes the proof.
\end{proof}

\section{Numerical solutions}\label{sec:numerical}
In this section, we propose numerical solutions to \eqref{Amaineqn_deterministic} 
and \eqref{Amaineqn} in the case of  generalized Kirchhoff vertex conditions. The approach combines a FEM discretization with a quadrature approximation of the fractional operator. 

\subsection{Variational formulation}\label{subsec:varform}
As a first step toward defining a numerical approximation to \eqref{Amaineqn_deterministic}, we first consider the nonfractional problem 
\begin{align}\label{eq:det}
	L_\alpha u=f,
\end{align}
where $f\in L_2(\Gamma)$, $\alpha\in\mathbb{R}$, and $L_{\alpha}$ 
is the operator \eqref{eq:operatorL} endowed with the generalized Kirchhoff vertex conditions \eqref{eq:kirchhoff}. Further, either let Assumption \ref{assum1_alt} hold, or let Assumption~\ref{assum1} hold, where 
$S = \frac{\lvert \alpha\rvert}{d_0}$ (see Remark~\ref{rmk1}).
By Theorem~\ref{compactLinverse}, \eqref{eq:det} has a unique solution $u=L_\alpha^{-1}f$, which is also a weak solution satisfying the variational problem
\begin{equation}\label{varform}
	\mathfrak{h}_\alpha(u,v)=(f,v),\quad\forall v\in H^{1}(\Gamma),
\end{equation}	
where the bilinear form $\mathfrak{h}_\alpha : H^{1}(\Gamma) \times H^{1}(\Gamma) \rightarrow \mathbb{R}$ is defined by
\begin{equation}\label{formL}
		\mathfrak{h}_\alpha(f,g)=  (\kappa^2 f, g)+\sum_{e\in\mathcal{E}}\int_{e} \H(s) f'(s) \cdot g'(s) ds\\
		+ \sum_{v\in\mathcal{V}} \frac{\alpha}{d_v}\langle F_f(v), F_g(v)\rangle,
\end{equation}
and $F_h(v) = (h_1(v),\ldots,h_{d_v}(v))^\top$.

We observe that Proposition~\ref{prp:Hdot2} provides a standard elliptic regularity. 
In fact, from the equivalence of norms of $H^2_\alpha(\Gamma)$ and 
$\dot{H}^2_{L_\alpha}(\Gamma)$, where $H^2_\alpha(\Gamma)$ was defined in 
the statement of Proposition~\ref{prp:Hdot2}, $C>0$ exists such that
\begin{align}\label{ellreg}
	\|u\|_{\widetilde{H}^2(\Gamma)}\leq C\|L_\alpha u\|_{L_2(\Gamma)} = C\|f\|_{L_2(\Gamma)}.
\end{align}

Recall that if $\alpha=0$, we obtain the Kirchhoff vertex conditions, and $S=0$, so
all the additional conditions are automatically satisfied.

\subsection[The finite dimensional space Vh]{The finite dimensional space $V_h$}
We construct a FEM discretization of the problem to numerically approximate the solution based on continuous piecewise linear basis functions, as in \cite{Arioli2017FEM}. Specifically, we create a regular subdivision of each edge $e$, with $n_e\geq 2$ intervals of length $h_e$. We let $\{x_{j}^{e}\}_{j=1}^{n_e-1}$ denote the interior nodes in this subdivision and construct standard hat basis functions on the edge:
\begin{align*}
	\psi_{j}^{e}(x^{e})=\begin{cases}
		1-\frac{\lvert x_{j}^{e}-x^e\rvert}{h_e}~~~~~\mbox{if} ~x_{j-1}^{e}\leq x^e\leq x_{j+1}^{e}\\
		0~~~~\hspace{1cm} \mbox{otherwise},
	\end{cases}
\end{align*}
where $j=1,2,\cdots n_e-1$, $x_{0}^{e}=0$, and $x_{n_e}^{e}=l_e$. These functions form a basis for
$$
V_{h}^{e}=\Big\{w\in H^{1}_{0}(e), w\lvert_{[x_{j}^{e},x_{j+1}^{e}]}\in \mathbb{P}^1, ~~j=0,1,\cdots, n_e-1\Big\},
$$
where $\mathbb{P}^1$ is the space of linear functions on $e$. To connect the edges in the finite element approximation, we also require basis functions centered on the vertices in the graph. To define these, let $e\in\mathcal{E}$ be an edge, $e = [0,l_e]$ us introduce the notation $\underline{e}$ as the vertex of $e$ at the position $0$, and $\bar{e}$ as the vertex of $e$ at the position $l_e$.  We define a neighboring set $\mathbb{W}_v$ for a vertex $v\in\mathcal{V}$ as 
\begin{align*}
	\mathbb{W}_v=\Bigg\{\bigcup_{e\in\mathcal{E}_v: ~\underline{e}=v}[v,x_{1}^{e}]\Bigg\}\cup \Bigg\{\bigcup_{e\in\mathcal{E}_v:~\bar{e}=v}[x_{n_e-1}^{e},v]\Bigg\}.	
\end{align*} 
Thus, we can define the hat functions centered on the vertices, with compact support $\supp(\phi_v(x))=\mathbb{W}$, as follows: 
\begin{align*}
	\phi_v(x^e)=\begin{cases}
		1-\frac{\lvert x_{v}^{e}-x^e\rvert}{h_e}~~~\mbox{if}~ x^e\in \mathbb{W}\cap e;  e\in\mathcal{E}_v\\
		0~~~\hspace{1cm}~~~\mbox{otherwise},
	\end{cases}
\end{align*}
where $x_{v}^{e}$ is equal to $0$ or $l_e$ depending on the direction on the edge and its parametrizations. 
Together, the functions $\psi_{j}^{e}$ and $\phi_{v}$ form the basis for the required finite dimensional space $V_{h}\subset H^{1}(\Gamma)$. Specifically, we define ${V_h=(\oplus_{e\in\mathcal{E}} V_h^e)\oplus V_v}$, where $V_v:=span(\{\phi_{v}, v\in\mathcal{V}\})$. In addition, the dimension of $V_h$, $N_h:=dim(V_h)$, is given by 
$N_h = \lvert\mathcal{V}\rvert +  \sum_{e\in\mathcal{E}}(n_e-1).$

\subsection{Finite element discretization}
We define $L_{h}:V_h\rightarrow V_h$ as the discrete version of $L_\alpha$:
\begin{equation}\label{eq:lh}
	(L_{h} u,v)= \mathfrak{h}_\alpha(u, v)\quad u, v \in V_h.
\end{equation}
The eigenvectors $\{e_{j,h}\}_{j=1}^{N_h}$ of $L_h$ satisfy the following variational equalities:
\begin{align*}
	(L_{h}e_{j,h},v)=\lambda_{j,h}(e_{j,h},v)~~~~~\forall v\in V_h, ~~~1\leq j\leq N_h,
\end{align*}
where $\{\lambda_{j,h}\}_{j=1}^{N_h}$ are eigenvalues of the corresponding 
eigenvectors $\{e_{j,h}\}_{j=1}^{N_h}$, and these eigenvalues satisfy 
$0<\lambda_{1,h}\leq \lambda_{2,h}\cdots \leq\lambda_{N_h,h}$.
By the min-max principle, 
\begin{equation}\label{eq:ineqminmax}
	\lambda_{j}\leq \lambda_{j,h},\quad j\in\mathbb{N},
\end{equation}
where $(\lambda_{j})_{j\in\mathbb{N}}$ are the eigenvalues of $L_\alpha$.
The finite element approximation is based on the variational formulation \eqref{varform} and reads: Find $u_h\in V_h$ such that
\begin{equation}\label{varformFEM}
	\mathfrak{h}_\alpha(u_h,v_h)=(f,v_h),\quad\text{ for all } v_h\in V_h.
\end{equation}
Alternatively, this can be stated as follows: Find $u_h\in V_h$ such that $L_{h}u_h=P_hf$, where ${P_h:L_2(\Gamma)\to V_h}$ denotes the $L_2(\Gamma)$-orthogonal projection onto $V_h$.

Throughout this section we will use the notation $\hat{h} = \max_e h_e$.

\begin{Proposition}\label{prp:rh}
	Let $\alpha\in\mathbb{R}$ and either assume Assumption \ref{assum1_alt}, or assume that Assumption~\ref{assum1} holds with 
	${S = \frac{\lvert \alpha\rvert}{d_0}}$ (see Remark~\ref{rmk1}). In addition, let $u$ be the solution of \eqref{varform} and $u_h$ be the solution of \eqref{varformFEM} for some $f\in L_2(\Gamma)$. Then, the following basic estimate holds:
	\begin{align}\label{Arioli1}
		\|u-u_h\|_{H^1({\Gamma})}\leq C\hat{h}\sum_{e\in E} \|u\lvert_e\|_{{H}^2(e)}=C\hat{h}\|u\|_{\widetilde{H}^2(\Gamma)},
	\end{align}
	where constant $C$ does not depend on $\hat{h}$.
\end{Proposition}

\begin{proof}
	The proof is essentially the same as the proof of \cite[Theorem 3.2]{Arioli2017FEM}, with the main difference being the regularity theorem we require due to the more general setup. However, we provide details on the various parts for completeness.
	
	We recall the bilinear form $\mathfrak{h}_\alpha$ defined in \eqref{formL}. Together, the assumptions with \eqref{vertexcal} and \eqref{positivedefinite2} in Theorem~\ref{compactLinverse} readily imply that $C_0, C_1>0$ exist such that
	$$
	C_0 \|z\|_{H^1(\Gamma)}^2 \leq \mathfrak{h}_\alpha(z,z)\quad\hbox{and}\quad \lvert\mathfrak{h}_\alpha(z,w)\rvert\leq C_1 \|z\|_{H^1(\Gamma)} \|w\|_{H^1(\Gamma)}.
	$$
	This yields that the bilinear form $\mathfrak{h}_\alpha$ induces a norm on $H^1(\Gamma)$ equivalent to the standard norm $\|\cdot\|_{H^1(\Gamma)}$.
	Furthermore, we let $V_\alpha:V_h\to \mathbb{R}$ be given by ${V_\alpha(z) = \mathfrak{h}_\alpha(u-z,u-z)}$. The fact that $V_h\subset H^1(\Gamma)$,
	along with \eqref{varform} and \eqref{varformFEM}, implies that $V_\alpha$ is minimized at $z=u_h$. Next, we let $u_h^I$ be the interpolant of $u$ in the nodes and vertices. Thus, 
	$$
	C_0 \|u-u_h\|_{H^1(\Gamma)}^2 \leq V_\alpha(u_h) \leq V_\alpha(u_h^I)  = \mathfrak{h}_\alpha(u-u_h^I,u-u_h^I) \leq C_1 \|u-u_h^I\|_{H^1(\Gamma)}^2.
	$$
	Subsequently, we observe that, by \eqref{varform}, $u$ solves the equation
	$L_\alpha u = f$; hence, $u\in \dot{H}^2_{L_\alpha}(\Gamma)$. Therefore, 
	by Proposition~\ref{prp:Hdot2}, in particular, it follows that
	$u\in \widetilde{H}^2(\Gamma)$, which implies that
	for each $e\in\mathcal{E}$, $u\lvert_e\in H^2(e)$. 
	The remaining proof is identical to its counterpart in \cite[Theorem 3.2]{Arioli2017FEM}.
\end{proof}

By introducing the Rayleigh--Ritz projection $R_{h,\alpha}:H^1(\Gamma)\to V_h$ as
$$ 
\mathfrak{h}_\alpha(u-R_{h,\alpha}u,v_h)=0,\quad \text{ for all } ~v_h\in V_h,
$$
it is apparent that $u_h=R_{h,\alpha}u$; hence, \eqref{Arioli1} can be formulated as follows:
$$
\|u-R_{h,\alpha}u\|_{H^1({\Gamma})}\leq C\hat{h}\|u\|_{\widetilde{H}^2(\Gamma)},\text{ for all }u\in\widetilde{H}^2(\Gamma )\cap H^1({\Gamma}).
$$
Moreover, $\dot{H}^2({\Gamma})\subset \widetilde{H}^2(\Gamma )\cap H^1({\Gamma})$. In a standard way, using the elliptic regularity and the Aubin--Nietsche duality trick, we obtain an optimal error bound in $L_2(\Gamma)$ stated in the following proposition whose proof we omit. 

\begin{Proposition} \label{prp:nt}
	Let $\alpha\in\mathbb{R}$ and either assume Assumption \ref{assum1_alt}, or assume that Assumption~\ref{assum1} holds with 
	${S = \frac{\lvert \alpha\rvert}{d_0}}$ (see Remark~\ref{rmk1}). 
	For all $u\in\widetilde{H}^2(\Gamma )\cap H^1({\Gamma})$, the following error estimate holds:
	\begin{align}\label{L2error}
		\|u-R_{h,\alpha}u\|_{L_2({\Gamma})}\leq C\hat{h}^2 \|u\|_{\widetilde{H}^2(\Gamma)}.
	\end{align}
	In particular, it holds for all $u\in \dot{H}_{L_\alpha}^2({\Gamma})\subset \widetilde{H}^2(\Gamma )\cap H^1({\Gamma})$.
\end{Proposition}

\begin{Remark}
	As for $u\in L_2(\Gamma)$, the $L_2(\Gamma)$-orthogonal projection $P_hu$ of $u$ is the best approximation of $u$ in the $L_2(\Gamma)$-norm; thus, we immediately conclude by Proposition~\ref{prp:Hdot2} that, for $u\in\dot{H}_{L_\alpha}^2({\Gamma})=D(L_\alpha)$, the following error estimate holds: 
	$$
	\|u-P_hu\|_{L_2(\Gamma)}\leq  C\hat{h}^2 \|u\|_{\widetilde{H}^2(\Gamma)}\leq C \hat{h}^2 \|L_\alpha u\|_{L_2(\Gamma)}.
	$$
	Moreover, because $P_h$ is a $L_2(\Gamma)$-contraction, using standard interpolation, 
	for ${u\in \dot{H}^{2s}({\Gamma})=D(L_\alpha^s)}$, $0\leq s\leq 1$, the following error estimate holds:
	\begin{equation}\label{eq:phest}
		\|u-P_hu\|_{L_2(\Gamma)}\leq C \hat{h}^{2s} \|L_\alpha^su\|_{L_2(\Gamma)}.
	\end{equation}
\end{Remark}

Next note that we can rewrite \eqref{L2error} in terms of the operators $L_\alpha$ and $L_{h}$ as follows:
\begin{equation}\label{eq:inverses}
	\|L_{h}^{-1}P_hf-L_\alpha^{-1}f\|_{L_2(\Gamma)}\leq C 
	\hat{h}^2 \|L_\alpha^{-1}f\|_{\widetilde{H}^2(\Gamma)}\leq C \hat{h}^2 \|f\|_{L_2(\Gamma)},
\end{equation}
where we applied the elliptic regularity estimate \eqref{ellreg} for the last inequality. 

The next result is analogous to \cite[Theorem 1]{coxkirchner} for the case of Euclidean domains. However there are key differences. 
By employing parabolic techniques in the proof, which is different to the strategy in \cite{coxkirchner}, 
we directly obtain a bound relative to the $\dot{H}_L^s(\Gamma)$ spaces (which is the norm of interest for us and, in general, often is of interest for stochastic problems) without requiring the equivalence between the $\dot{H_L^s}$ norm and the fractional Sobolev norm 
(this was required in \cite{coxkirchner}).  By doing so, we are able to prove this result under the generalized Kirchhoff conditions. 
Finally, in contrast to  \cite[Theorem 1]{coxkirchner}, we do not require the stability of the $L_2(\Gamma)$-orthogonal projection onto the finite element space in the $H^1(\Gamma)$-norm nor the availability of an inverse inequality; both of these require additional restrictions on the mesh. 
\begin{Lemma}\label{lem:fracp}
	Let $\alpha\in\mathbb{R}$ and either assume Assumption \ref{assum1_alt}, or assume that Assumption~\ref{assum1} holds with 
	${S = \frac{\lvert \alpha\rvert}{d_0}}$ (see Remark~\ref{rmk1}). 
	\begin{itemize}
		\item[(a)] Let $0<\beta <1$, $0\leq s \leq 1$, and $f\in D(L_\alpha^s)$. 
		If $\beta+s<1$, then for all sufficiently small $\epsilon>0$, $C=C(\epsilon,s,\beta)$ exists such that, for $0<\hat{h}<1$,
		$$
		\|L_{h}^{-\beta}P_hf-L_\alpha^{-\beta}f\|_{L_2(\Gamma)}\leq C  \hat{h}^{2\beta+2s-\epsilon}\|L_\alpha^sf\|_{L_2(\Gamma)} = C \hat{h}^{2\beta+2s-\epsilon}\|f\|_{\dot{H}_{L_\alpha}^{2s}(\Gamma)}.
		$$
		If $\beta+s\geq 1$, then for all sufficiently small $\epsilon>0$, $C=C(\epsilon,s,\beta)$ exists such that, for $0<\hat{h}<1$,
		$$
		\|L_{h}^{-\beta}P_hf-L_\alpha^{-\beta}f\|_{L_2(\Gamma)}\leq C  \hat{h}^{2-\epsilon}\|L_\alpha^sf\|_{L_2(\Gamma)} = C  \hat{h}^{2-\epsilon} \|f\|_{\dot{H}^{2s}_{L_\alpha}(\Gamma)}.
		$$ 
		\item[(b)] If $1<\beta<2$, then for all sufficiently small $\epsilon>0$, $C=C(\epsilon,\beta)$ exists such that
		$$
		\|L_{h}^{-\beta}P_hf-L_\alpha^{-\beta}f\|_{L_2(\Gamma)}\leq C  \hat{h}^{2-\epsilon}\|f\|_{L_2(\Gamma)},\quad 0<\hat{h}<1.
		$$
	\end{itemize}
\end{Lemma}
\begin{proof}
	(a) As $L_\alpha$ is self-adjoint and positive definite on $L_2(\Gamma)$, it 
	follows that $-L_\alpha$ generates an analytic semigroup of contractions 
	$\{E(t)\}_{t\geq 0}$ satisfying
	$$
	\|\frac{d^n}{dt^n}(E(t)f)\|_{L_2(\Gamma)}\leq \frac{C_n}{t^n}\|f\|_{L_2(\Gamma)},\quad t>0,
	$$ 
	for some $C_n>0$, $n=1,2,\dots$. The operator $T_h:=L_{h}^{-1}P_h$ is self-adjoint and positive semidefinite on $L_2(\Gamma)$ and positive definite on $V_h$. Therefore, considering \eqref{eq:inverses}, we can copy the proof of \cite[Theorem 3.5]{Thomee} to conclude that
	$$
	\|E_{h}(t)P_hf-E(t)f\|_{L_2(\Gamma)}\leq C\hat{h}^2t^{s-1}\|L_\alpha^sf\|_{L_2(\Gamma)},\quad f\in D(L_\alpha^s),\quad s\in [0,1],
	$$ 
	and
	$$
	\|E_{h}(t)P_hf-E(t)f\|_{L_2(\Gamma)}\leq C\hat{h}^{2s}\|L_\alpha^sf\|_{L_2(\Gamma)},\quad f\in D(L_\alpha^s),\quad s\in [0,1],
	$$
	where $\{E_{h}(t)\}_{t\geq 0}$ is the semigroup of contractions on $V_h$ generated by $-L_{h}$. Therefore, for $f\in D(L_\alpha^s)$ and $s,\gamma\in [0,1]$,
	\begin{align*}
		\|E_{h}(t)P_hf-E(t)f\|_{L_2(\Gamma)} &\leq C\hat{h}^{2\gamma} t^{\gamma(s-1)}\hat{h}^{2s(1-\gamma)}\|L_\alpha^sf\|_{L_2(\Gamma)}\\
		&=C t^{\gamma(s-1)}\hat{h}^{2(s+\gamma-s\gamma)}\|L_\alpha^sf\|_{L_2(\Gamma)}.
	\end{align*}
	Next, for $0< \beta<1$ and $0\leq s \leq 1$, we write
	\begin{align*}
		&\|L_{h}^{-\beta}P_hf-L_\alpha^{-\beta}f\|_{L_2(\Gamma)}=\left\| \frac{1}{\Gamma(\beta)}\int_{0}^{\infty}t^{\beta-1}(E_{h}(t)P_hf-E(t)f)\,dt\right\|_{L_2(\Gamma)}\\
		&\leq \int_0^{\frac1{\hat{h}^2}}\frac{t^{\beta-1}}{\Gamma(\beta)}\|E_{h}(t)P_hf-E(t)f\|_{L_2(\Gamma)}\,dt
		+\int_{\frac1{\hat{h}^2}}^{\infty}\frac{t^{\beta-1}}{\Gamma(\beta)}\|E_{h}(t)P_hf-E(t)f\|_{L_2(\Gamma)}\,dt\\
		&=:E_1+E_2,
	\end{align*}
	where we used the representation of negative fractional powers from \cite[Chapter~2, (6.9)]{Pazy} in the first equality above. We estimate the first term as follows:
	\begin{align}
		E_1 &\leq  C\int_0^{\hat{h}^{-2}} t^{\gamma(s-1)+\beta -1}\hat{h}^{2(s+\gamma-s\gamma}\|L_\alpha^sf\|_{L_2(\Gamma)}\,dt\notag \\
		&\leq C \hat{h}^{4\gamma(1-s)+2s-2\beta}\|L_\alpha^sf\|_{L_2(\Gamma)}, \label{eq:firstterm}
	\end{align}
	provided that $\gamma(s-1)+\beta -1>-1$ (i.e., $\gamma<\frac{\beta}{1-s}$). Let now $\beta+s<1$ (i.e., $\frac{\beta}{1-s}<1$) and for sufficiently small $\epsilon>0$ set $\gamma=\frac{\beta-\nicefrac{\epsilon}{4}}{1-s}$. Then, from \eqref{eq:firstterm},  
	${E_1\leq C \hat{h}^{2\beta+2s-\epsilon}\|L_\alpha^sf\|_{L_2(\Gamma)}}$ and
	for all $0<\beta<1$ and $0\leq s \leq 1$, we have
	\begin{align*}
		E_2 &\leq C\hat{h}^2\int_{\hat{h}^{-2}}^\infty 	t^{\beta-1-1}\|f\|_{L_2(\Gamma)}\leq C \hat{h}^{4-2\beta}\|f\|_{L_2(\Gamma)}\\
		&\leq C \hat{h}^2 \|f\|_{L_2(\Gamma)}
		\leq C \hat{h}^2 \|L_\alpha^sf\|_{L_2(\Gamma)}.
	\end{align*}
	Thus, if $0< \beta<1$, $0\leq s \leq 1$, and $\beta+s<1$, then for all $f\in D(L_\alpha^s)$,
	$$
	\|L_{h}^{-\beta}P_hf-L_\alpha^{-\beta}f\|_{L_2(\Gamma)}\leq C  \hat{h}^{2\beta+2s-\epsilon}\|L_\alpha^sf\|_{L_2(\Gamma)}.
	$$
	for sufficiently small $\epsilon>0$. Let now $0<\beta <1$, $0 \leq s \leq 1$, such that $\beta+s\geq 1$ and $f\in D(L_\alpha^s)$. 
	Fix some arbitrary $\delta$ such that $0<\delta<2-2\beta$. Then, $0\leq \tilde{s}<s$  exists
	such that $\beta+\tilde{s}=1-\nicefrac{\delta}{2}$. Thus, $f\in D(L_\alpha^{\tilde{s}})$ and
	\begin{equation*}
		\|L_{h}^{-\beta}P_hf-L_\alpha^{-\beta}f\|_{L_2(\Gamma)}\leq C  \hat{h}^{2-\delta-\epsilon}\|L_\alpha^{\tilde{s}}f\|_{L_2(\Gamma)}
		\leq C\hat{h}^{2-\delta-\epsilon}\|L_\alpha^{s}f\|_{L_2(\Gamma)}.
	\end{equation*}
	(b) Let $S(t):=e^{-tL_\alpha^2}$ and $S_h(t):=e^{-tL_{h}^2}$ (i.e., the operator semigroups generated by $-L_\alpha^2$ and $-L_{h}^2$, 
	respectively). As mentioned above, both semigroups are analytic. Using Parseval's identity reveals that
	$
	\|S(t)f\|_{L_2(\Gamma)}\leq e^{-\lambda^2_1 t}\|f\|_{L_2(\Gamma)}$ for $t\geq 0$,
	and
	$$
	\|S_h(t)P_hf\|_{L_2(\Gamma)}\leq e^{-\lambda^2_{1,h} t}\|f\|_{L_2(\Gamma)}\leq e^{-\lambda_1^2 t}\|f\|_{L_2(\Gamma)},\quad t\geq 0.
	$$
	Given Propositions~\ref{prp:rh} and \ref{prp:nt}, we use the same abstract argument that proves \cite[Corollary 5.3]{EL92} to conclude that
	$\|S_h(t)P_hf-S(t)f\|_{L_2(\Gamma)}\leq C \hat{h}^2t^{-\frac12}\|f\|_{L_2(\Gamma)}$, for $t>0$ and $0<\hat{h}<1$.
	Thus, for arbitrary $0<\epsilon \leq 1$, 
	$$
	\|S_h(t)P_hf-S(t)f\|_{L_2(\Gamma)}\leq C \hat{h}^{2(1-\epsilon)}t^{-\frac{1}{2}(1-\epsilon)}e^{-\epsilon \lambda_1^2 t}\|f\|_{L_2(\Gamma)},\quad t>0,\quad 0<\hat{h}<1.
	$$
	Therefore, if $1<\beta<2$, then
	\begin{align*}
		\|L_{h}^{-\beta}P_hf-& L_\alpha^{-\beta}f\|_{L_2(\Gamma)} =
		\|(L_{h}^2)^{-\frac{\beta}{2}}P_hf-(L_\alpha^2)^{-\frac{\beta}{2}}f\|_{L_2(\Gamma)}\\
		&=\left\| \frac{1}{\Gamma(\beta)}\int_{0}^{\infty}t^{\frac{\beta}{2}-1}(S_h(t)P_hf-S(t)f)\,dt\right\|_{L_2(\Gamma)}\\
		&\leq C\hat{h}^{2(1-\epsilon)}\int_{0}^{\infty}t^{\frac{\beta}{2}-1-\frac{1}{2}(1-\epsilon)}e^{-\lambda_1^2\epsilon t}\,dt \|f\|_{L_2(\Gamma)}
		\leq C \hat{h}^{2(1-\epsilon)}\|f\|_{L_2(\Gamma)},
	\end{align*}
	and the proof of part (b) and, thus, that of the lemma is complete.
\end{proof}

\subsection{The fractional deterministic model}\label{sec:deterministic}
Next, we construct the numerical approximation of the solution $u$ in \eqref{Amaineqn_deterministic}. To obtain the final approximation, 
we must address the fractional power $L_{h}^\beta$. We do so via a
quadrature approximation of the inverse fractional power operator 
by applying Dunford--Taylor calculus (see \cite{balakrishnan1960fractional}). 
Consider the following exponentially convergent quadrature approximation 
of $L_{h}^{-\beta}$:  
\begin{align}\label{quadrature}
	Q_{h,k,\beta}:=\frac{2k \sin(\pi\beta)}{\pi}\sum_{l=-K^{-}}^{K^+}e^{2\beta l k}\left(Id_{V_h}+e^{2l k}L_{h}\right)^{-1},
\end{align} 
as introduced by \cite{bonito2015numerical}, where $k>0$ is the quadrature step size, $K^{-}:=\ceil[\big]{\frac{\pi^2}{4\beta k^2}}$, and 
$K^{+}:=\ceil[\big]{\frac{\pi^2}{4(1-\beta) k^2}}.$
With \eqref{quadrature}, we define an approximation of the solution to \eqref{Amaineqn_deterministic} as
\begin{align}\label{mainApproximation_determ}
	u_{h,k}^{Q}:=Q_{h,k,\beta}f_h,
\end{align}
where $f_h = P_h f$. 
The following theorem provides the strong convergence of the approximation
$u_{h,k}^{Q}$ in \eqref{mainApproximation_determ} to the solution $u$ of \eqref{Amaineqn_deterministic}.

\begin{Theorem}\label{thm:strong_rate_determ}
	Let $\alpha\in\mathbb{R}$ and either assume Assumption \ref{assum1_alt}, or assume that Assumption~\ref{assum1} holds with 
	${S = \frac{\lvert \alpha\rvert}{d_0}}$ (see Remark~\ref{rmk1}). 
	If $0<\beta <1$ and $\sigma<\beta$, then $C=C_\sigma>0$ exists such that
	$
	\|u-u_{h,k}^{Q}\|_{L_2(\Gamma)}\leq C(\hat{h}^{2\sigma}+e^{-\frac{\pi^2}{2k}})\|f\|_{L_2(\Gamma)}.
	$
\end{Theorem}

\begin{proof}
	We write
	$$
	u-u_{h,k}^{Q}=(L_\alpha^{-\beta}f-
	L_{h}^{-\beta}P_hf_h)+(L_{h}^{-\beta}P_h-
	Q_{h,k,\beta})f_h:=E_1+E_2.
	$$	
	To bound $E_1$, we use Lemma~\ref{lem:fracp} to obtain
	$\|E_1\|_{L_2(\Gamma)}^2\leq C \hat{h}^{2\sigma}\|f\|^2_{L_2(\Gamma)}$, and 
	by \cite[Theorem 3.5]{bonito2015numerical},  
	$\|E_2\|_{L_2(\Gamma)}\leq C e^{-\frac{\pi^2}{2k}}\|f_h\|_{L_2(\Gamma)}\leq C e^{-\frac{\pi^2}{2k}} \|f\|_{L_2(\Gamma)}$.
\end{proof}

\subsection{The fractional stochastic model}\label{sec:fem_stochastic}
At this point, we can construct the numerical approximation of the generalized Whittle--Mat\'ern field  $u$ defined in \eqref{Amaineqn}. To that end, first, we perform the same finite 
element approximation as in the deterministic case, resulting in the discrete 
problem of finding $u_h\in V_h$ such that $L_{h}^\beta u_h= \mathcal{W}_h$, where
$\mathcal{W}_h$ is a $V_h$ valued approximation of the white noise defined as
	$\mathcal{W}_h:=\sum_{k=1}^\infty \xi_kP_he_k.$

To obtain 
the final approximation, we proceed as in Section~\ref{sec:deterministic} and apply 
$Q_{h,k,\beta}$ instead of $L_{h}^{-\beta}$. 
Applying the quadrature scheme, we define an approximation of the solution 
$u$ to \eqref{Amaineqn} as follows:
\begin{align}\label{mainApproximation}
	u_{h,k}^{Q}:=Q_{h,k,\beta}\mathcal{W}_{h}.
\end{align}

	Let $\{\phi_i\}_{i=1}^{N_h}$ denote a basis for $V_h$. To sample the approximate solution \eqref{mainApproximation}, we must sample the vector
	$
	\mv{W}_h=[\langle\mathcal{W}_h,\phi_1\rangle_{L_2(\Gamma)}, \langle\mathcal{W}_h,\phi_12\rangle_{L_2(\Gamma)},\dots, \langle\mathcal{W}_h,\phi_{N_h}\rangle_{L_2(\Gamma)} ]^\top,
	$ 
	which is a zero mean $N_h$-dimensional Gaussian vector with covariance matrix $\mv{M}$ with elements
	$M_{ij}=\langle \phi_i,\phi_j\rangle_{L_2(\Gamma)}$, $i,j=1,2,\dots, N_h.$ 
	Indeed, $\mv{W}_h$ is a zero mean Gaussian, as $\mathcal{W}_h$ is a zero mean Gaussian in $L_2(\Gamma)$. Using Parseval’s formula, if $\psi,\phi\in V_h$, then
	$$
	\mathbb{E} (\langle \mathcal{W}_h,\phi  \rangle_{L_2(\Gamma)} \langle \mathcal{W}_h,\psi  \rangle_{L_2(\Gamma)} )=\sum_{k=1}^\infty \langle P_he_k,\phi \rangle_{L_2(\Gamma)}  \langle P_he_k,\psi \rangle_{L_2(\Gamma)} =
	\langle \phi,\psi\rangle_{L_2(\Gamma)} .
	$$
	Thus, for $\mv{a},\mv{b}\in \mathbb{R}^{N_h}$, we obtain
	\begin{equation}\label{eq:sample}
		\begin{aligned}
		\mathbb{E}(\langle \mv{W}_h, \mv{a}\rangle_{\mathbb{R}^{N_h}} \langle \mv{W}_h, \mv{b}\rangle_{\mathbb{R}^{N_h}})
		&=\mathbb{E}  \Big(\big\langle \mathcal{W}_h,\sum_{i=1}^{N_h}a_i\phi_i  \big\rangle_{L_2(\Gamma)} \big\langle \mathcal{W}_h,\sum_{i=1}^{N_h}b_i\phi_i  \big\rangle_{L_2(\Gamma)} \Big)\\
		&=\big\langle \sum_{i=1}^{N_h}a_i\phi_i, \sum_{i=1}^{N_h}b_i\phi_i \big\rangle_{L_2(\Gamma)}=\big\langle \mv{M} \mv{a},\mv{b}\big\rangle_{\mathbb{R}^{N_h}}.
		\end{aligned}
	\end{equation}

The following theorem provides the strong convergence of the approximation
$u_{h,k}^{Q}$ in \eqref{mainApproximation} to the solution $u$ of \eqref{Amaineqn}.

\begin{Theorem}\label{thm:strong_rate}
	Let $\alpha\in\mathbb{R}$ and either assume Assumption \ref{assum1_alt}, or assume that Assumption~\ref{assum1} holds with 
	${S = \frac{\lvert \alpha\rvert}{d_0}}$ (see Remark~\ref{rmk1}). 
	In addition, let $u$ be the solution of \eqref{Amaineqn} under the generalized Kirchhoff vertex conditions. If $\frac14<\beta <1$ and $\frac14<\sigma<\beta$, then $C=C_\sigma>0$ exists such that
	$$
	(\mathbb{E}\|u-u_{h,k}^{Q}\|_{L_2(\Gamma)}^2)^{\frac12}\leq C(\hat{h}^{2\sigma-\frac12}+\hat{h}^{2\sigma}N_h^\frac12+e^{-\frac{\pi^2}{2k}}N_h^{\frac12}),
	$$
	where $N_h=\dim V_h$.
\end{Theorem}

\begin{proof}
	We write
	\begin{align*}
		u-u_{h,k}^{Q}&=L_\alpha^{-\beta}(\mathcal{W}-\mathcal{W}_h)+(L_\alpha^{-\beta}-
		L_{h}^{-\beta}P_h)\mathcal{W}_h+(L_{h}^{-\beta}P_h-
		Q_{h,k,\beta})\mathcal{W}_h \\
		&:=E_1+E_2+E_3.
	\end{align*}
	For $E_1$, using the fact that $L_\alpha^{-\beta}$ and 
	$(I-P_h)$ are self-adjoint, the estimate \eqref{eq:phest} and the 
	asymptotics \eqref{eq:eiga}, we find
	\begin{align*}
		\mathbb{E}\|E_1\|_{L_2(\Gamma)}^2 &= \sum_{\ell=1}^\infty\|L_\alpha^{-\beta}(I-P_h)e_\ell\|^2=
		\|L_\alpha^{-\beta}(I-P_h)\|^2_{\mathcal{L}_2(L_2(\Gamma))}\\
		&=\|(I-P_h)L_\alpha^{-\beta}\|^2_{\mathcal{L}_2(L_2(\Gamma))} = \sum_{\ell=1}^\infty\|(I-P_h)L_\alpha^{-\beta}e_\ell\|_{L_2(\Gamma)}^2 \\
		&\leq C\hat{h}^{4s}\sum_{\ell=1}^\infty\|L_\alpha^{s-\beta}e_\ell\|^2_{L_2(\Gamma)}
		\leq C \hat{h}^{4s} \sum_{\ell=1}^\infty \ell^{4s-4\beta}\leq C_\sigma \hat{h}^{4\sigma-1},
	\end{align*}
	where we took $s=\sigma-1/4$. 
	To bound $E_2$, we use Lemma~\ref{lem:fracp} to obtain
	$$
	\mathbb{E}\|E_2\|_{L_2(\Gamma)}^2\leq C \hat{h}^{4\sigma}\mathbb{E}\|\mathcal{W}_h\|^2_{L_2(\Gamma)}= C \hat{h}^{4\sigma} \|P_h\|^2_{\mathcal{L}_2(L_2(\Gamma))}= C \hat{h}^{4\sigma} N_h.
	$$
	We similarly bound $E_3$, this time using \cite[Theorem 3.5]{bonito2015numerical} 
	for the quadrature error $L_{h}^{-\beta}P_h-Q_{h,k,\beta}$,
	$
	\mathbb{E}\|E_3\|_{L_2(\Gamma)}^2\leq C e^{-\frac{\pi^2}{k}}\mathbb{E}\|\mathcal{W}_h\|^2_{L_2(\Gamma)}\leq C e^{-\frac{\pi^2}{k}} N_h,
	$.
\end{proof}

\begin{Remark}\label{rem:calibration}
	If we use a quasiuniform mesh so that $N_h \propto h^{-1}$ and calibrate $k$ so that $e^{-\frac{\pi^2}{2k}}\propto \hat{h}^{2\sigma}$, we obtain the strong convergence rate  $2\sigma - \nicefrac12$.
\end{Remark}

\begin{Remark}\label{rem:quad}
	The quadrature approximation $Q_{h,k,\beta}$ in \eqref{mainApproximation} can be replaced with any other exponentially convergent rational 
	approximation of the fractional operator (e.g., that proposed in \cite{BK2020rational}) without changing the rate of 
	convergence in Theorem~\ref{thm:strong_rate}. 
\end{Remark}

If $u$ solves $L_\alpha^\beta u = \mathcal{W}$, then the covariance operator
of $u$ is $L_\alpha^{-2\beta}$, and if we let
$
\varrho^\beta(x,y) = \sum_{j=1}^{\infty} \lambda_{j}^{-2\beta} e_{j}(x) e_{j}(y),\hbox{ for a.e. $(x,y)\in\Gamma\times\Gamma$},
$ 
then $\varrho^\beta$ is the covariance function of  $u$ and the
kernel of the operator $L_\alpha^{-2\beta}$. Likewise,
if $u_h$ is the solution of $L_{h,\alpha}^\beta u_h= \mathcal{W}_h$, then the covariance 
operator of $u_h$ is given by $L_{h}^{-2\beta}$. Furthermore, if we let
$
\varrho_h^\beta(x,y) = \sum_{j=1}^{N_h} \lambda_{j,h}^{-2\beta} e_{j,h}(x) e_{j,h}(y),\hbox{ for a.e. $(x,y)\in\Gamma\times\Gamma$},
$
then $\varrho_h^\beta$ is the covariance function of $u_h$ 
and the kernel of $L_{h}^{-2\beta}$. 
We now state and prove a result in which we obtain the convergence of the finite element approximation $ \varrho_h^\beta$ to the covariance function $\varrho^\beta$ of the solution $u$ of \eqref{Amaineqn} in the $L_2(\Gamma\times\Gamma)$-norm. 

\begin{Theorem}\label{cov_fem_approx_rate}
	Let $\alpha\in\mathbb{R}$ and either assume Assumption \ref{assum1_alt}, or assume that Assumption~\ref{assum1} holds with 
	${S = \frac{\lvert \alpha\rvert}{d_0}}$ (see Remark~\ref{rmk1}).  In addition, let $\beta>\nicefrac14$ and $\varrho^\beta$ denote the covariance function of the solution to \eqref{Amaineqn} under the generalized Kirchhoff vertex conditions. Then, for ${\sigma<\min\{4\beta-\nicefrac12,2\}}$, 
	$\|\varrho^\beta - \varrho_h^\beta\|_{L_2(\Gamma\times\Gamma)} \lesssim_{\sigma,\beta, \H, \kappa,\Gamma} \hat{h}^{\sigma}.$
\end{Theorem}

\begin{proof}
	We begin by observing that $L_\alpha^{-2\beta} - L_{h}^{-2\beta}P_h$ is an 
	integral operator with the kernel given by $\varrho^\beta - \varrho_h^\beta$. Thus,
	$\|\varrho^\beta - \varrho_h^\beta\|_{L_2(\Gamma\times\Gamma)} = 
	\|L_\alpha^{-2\beta} - L_{h}^{-2\beta}P_h\|_{\mathcal{L}_2(L_2(\Gamma))}.$
	Fix any $\varepsilon>0$ and let ${0<\delta < \min\{\beta -\nicefrac14, \nicefrac{\varepsilon}{4}\}}$. Then, with $\tau = 2\beta -\nicefrac14-\delta > 0$, using the same decomposition of the error as in \cite[Proposition 4 (Part IIa)]{coxkirchner}, we have
	\begin{equation}\label{proof1}
		\begin{split}
			\|L_\alpha^{-2\beta} - L_{h}^{-2\beta}P_h\|_{\mathcal{L}_2(L_2(\Gamma))} &\leq \left\|\left(L_\alpha^{-\tau} - L_{h}^{-\tau}P_h\right)L_{h}^{-(\nicefrac14+\delta)}P_h\right\|_{\mathcal{L}_2(L_2(\Gamma))}\\
			&\,\,\,\,\,+ \left\|L_\alpha^{-\tau}\left(L_h^{-(\nicefrac14+\delta)}P_h - L_\alpha^{-(\nicefrac14+\delta)}\right)\right\|_{\mathcal{L}_2(L_2(\Gamma))}.
		\end{split}
	\end{equation}
	
	By following a similar idea to the one in the proof of \cite[Proposition 4]{coxkirchner}, we obtain that
	for sufficiently small $h$, 
	\begin{equation}\label{bound1}
		\left\|\left(L_\alpha^{-\tau} - L_{h}^{-\tau}P_h\right)L_{h}^{-(\nicefrac14+\delta)}P_h\right\|_{\mathcal{L}_2(L_2(\Gamma))} \lesssim_{\varepsilon, \beta,\kappa, \H, \Gamma} \hat{h}^{\min\{4\beta-\nicefrac12-\varepsilon, 2-\nicefrac{\varepsilon}{2}\}}.
	\end{equation}
	
	To conclude the proof, a bound for the second term on the right-hand side of \eqref{proof1} must still be obtained. In this part our proof differs from the one in \cite{coxkirchner}, as Lemma \ref{lem:fracp} allows us to avoid requiring equivalence between the norms $\|\cdot\|_{\dot{H}^\gamma_{L_\alpha}(\Gamma)}$ and $\|\cdot\|_{H^\gamma(\Gamma)}$, which was a key ingredient in the proof given in  \cite{coxkirchner}. Therefore, we let $\gamma := \min\{4\beta-1-4\delta,2\}>0$, and we then have 
	\begin{equation}\label{norm_ineq}
		\begin{split}
			\left\|L_\alpha^{-\tau}\right.&\left.\left(L_{h}^{-(\nicefrac14+\delta)}P_h - 
			L_\alpha^{-(\nicefrac14+\delta)}\right)\right\|_{\mathcal{L}_2(L_2(\Gamma))} \\
			&\leq \left\|L_\alpha^{-\tau}\right\|_{\mathcal{L}_2(L_2(\Gamma),\dot{H}_{L_\alpha}^\gamma(\Gamma))}
			\left\|L_{h}^{-(\nicefrac14+\delta)}P_h - L_\alpha^{-(\nicefrac14+\delta)}\right\|_{\mathcal{L}(\dot{H}_{L_\alpha}^\gamma(\Gamma),L_2(\Gamma))}.
		\end{split}
	\end{equation}
	We now bound the two norms appearing in the above expression separately. Recall that $\{e_j\}_{j\in\mathbb{N}}$ is an orthonormal basis in $L_2(\Gamma)$. Therefore, by Corollary~\ref{cor:genkirchhoffweyl},
	\begin{align*}
		\left\|L_\alpha^{-\tau}\right\|_{\mathcal{L}_2(L_2(\Gamma),\dot{H}_{L_\alpha}^\gamma(\Gamma))}^2
		&=
		\sum_{j=1}^\infty \left\|L_\alpha^{-\tau}e_j\right\|_{\dot{H}_{L_\alpha}^\gamma(\Gamma)}^2
		=
		\sum_{j=1}^\infty \left\|L_\alpha^{\nicefrac{\gamma}{2}} L_\alpha^{-\tau}e_j\right\|_{L_2(\Gamma)}^2\\
		&=
		\sum_{j=1}^\infty \lambda_j^{\gamma-4\beta + \nicefrac12 + 2\delta}
		\lesssim_{\kappa, \H,\Gamma}
		\sum_{j=1}^\infty j^{2\gamma-8\beta+4\delta+1},
	\end{align*}
	which is convergent because $2\gamma-8\beta+4\delta+1< -1$, and the last inequality holds because $\gamma-2\delta<\gamma\leq 4\beta-1-4\delta$.
	
	Finally, we bound the second term on the right of \eqref{norm_ineq} 
	 using Lemma~\ref{lem:fracp} as:
	\begin{align*}
		\left\| L_{h}^{-(\nicefrac14+\delta)} P_h-L_\alpha^{-(\nicefrac14+\delta)} \right\|_{\mathcal{L}(\dot{H}_{L_\alpha}^\gamma(\Gamma),L_2(\Gamma))}
		&\lesssim_{\varepsilon,\delta,\gamma, \kappa, \H, \Gamma}
		\hat{h}^{\min\{2(\nicefrac14+\delta) + \gamma - \nicefrac{\varepsilon}{2},2-\nicefrac{\varepsilon}{2}\}}\\
		&=
		\hat{h}^{\min\{4\beta-\nicefrac12-2\delta-\nicefrac{\varepsilon}{2},2-\nicefrac{\varepsilon}{2}\}}.
	\end{align*}
	To conclude, we observe that $2\delta < \nicefrac{\varepsilon}{2}$. Thus, for sufficiently small $h$, 
	\begin{equation}\label{bound2}
		\left\| L_{h}^{-(\nicefrac14+\delta)} P_h-L_\alpha^{-(\nicefrac14+\delta)} \right\|_{\mathcal{L}(\dot{H}_{L_\alpha}^\gamma(\Gamma),L_2(\Gamma))}
		\lesssim_{\varepsilon, \beta,\kappa, \H, \Gamma} \hat{h}^{\min\{4\beta-\nicefrac12-\varepsilon, 2-\nicefrac{\varepsilon}{2}\}}.
	\end{equation}
	Thus, the proof follows from \eqref{bound1} and \eqref{bound2}.
\end{proof}

Finally, we consider the quadrature approximation
 $u_{h,k}^Q = Q_{h,k,\beta} \mathcal{W}_h$. The covariance 
operator of $u_{h,k}^Q$ is $Q_{h,k,2\beta}$. 
Let $q_{k,\beta}(x)=\frac{2k \sin(\pi\beta)}{\pi}\sum_{l=-K^{-}}^{K^+}e^{2\beta l k}(1+e^{2l k}x)^{-1}$.
Then, the kernel of $Q_{h,k,2\beta}$, and covariance function of $u_{h,k}^Q$, is
$$
\varrho_{h,k}^\beta(x,y) = \sum_{j=1}^{N_h} q_{k,2\beta}(\lambda_{j,h,\alpha}) e_{j,h,\alpha}(x) e_{j,h,\alpha}(y),\quad\hbox{for a.e. $(x,y)\in\Gamma\times\Gamma$}.
$$

\begin{Theorem}\label{cov_fem_approx_rate_quadrature}
	Let $\alpha\in\mathbb{R}$ and either assume Assumption \ref{assum1_alt}, or assume that Assumption~\ref{assum1} holds with 
	${S = \frac{\lvert \alpha\rvert}{d_0}}$ (see Remark~\ref{rmk1}). Additionally, fix $\beta>\nicefrac14$ and let $\varrho^\beta$ denote the covariance function of the solution to \eqref{Amaineqn} under the generalized Kirchhoff vertex conditions. Then, for any ${\sigma<\min\{4\beta-\nicefrac12,2\}}$, we have
	$$\|\varrho^\beta - \varrho_{h,k}^\beta\|_{L_2(\Gamma\times\Gamma)} \lesssim_{\sigma,\beta, \H, \kappa,\Gamma} \hat{h}^{\sigma} + e^{-\frac{\pi^2}{2k}}N_h^{\nicefrac12}.$$
\end{Theorem}

\begin{proof}
	First, by the triangle inequality, 
	$$\|\varrho_{h,m}^\beta - \varrho^\beta\|_{L_2(\Gamma\times\Gamma)} \leq \|\varrho_{h}^\beta - \varrho^\beta\|_{L_2(\Gamma\times\Gamma)} + \|\varrho_{h,m}^\beta - \varrho_h^\beta\|_{L_2(\Gamma\times\Gamma)}.$$
	Further, by Theorem~\ref{cov_fem_approx_rate},
	$\|\varrho_{h}^\beta - \varrho^\beta\|_{L_2(\Gamma\times\Gamma)} \lesssim_{\sigma,\beta, \H, \kappa,\Gamma} \hat{h}^{\sigma}.$
	Therefore, we must still bound $\|\varrho_{h,m}^\beta - \varrho_h^\beta\|_{L_2(\Gamma\times\Gamma)}$. We observe that
	$$
	\|\varrho_{h,m}^\beta - \varrho_h^\beta\|_{L_2(\Gamma\times\Gamma)} = \|Q_{h,k,2\beta} P_h - L_{h,\alpha,-2\beta}P_h\|_{\mathcal{L}_2(L_2(\Gamma))}.
	$$
	Furthermore, from Lemma~\ref{lem:fracp} and \cite[Theorem 3.5]{bonito2015numerical}, with the fact that the operator $L_\alpha^{-\beta}$ is bounded, some constant $C$ exists such that, for every $0 < \hat{h}<1$, $\|L_{h}^{-\beta}P_h\|_{\mathcal{L}(L_2(\Gamma))} \leq C$
	and $\|Q_{h,k,\beta}P_h\|_{\mathcal{L}(L_2(\Gamma))} \leq C.$
	Again using \cite[Theorem 3.5]{bonito2015numerical} for the quadrature error, we obtain
	\begin{align*}
		\|L_{h}^{-\beta}P_h - Q_{h,k,\beta} P_h&\|_{\mathcal{L}_2(L_2(\Gamma))}^2 =
		\|(L_{h}^{-\beta}P_h - Q_{h,k,\beta} P_h)P_h\|_{\mathcal{L}_2(L_2(\Gamma))}^2\\
		&\leq \|L_{h}^{-\beta}P_h - Q_{h,k,\beta} P_h\|_{\mathcal{L}(L_2(\Gamma))}^2\|P_h\|^2_{\mathcal{L}_2(L_2(\Gamma))}
		\leq C e^{-\frac{\pi^2}{k}} N_h.
	\end{align*}
	Finally, by gathering all these results and considering that 
	$L_{h}^{-\beta}$ and $ Q_{h,k,\beta}$ commute on $V_h$ and thus 
	$L_{h}^{-2\beta} P_h - Q_{h,k,2\beta}P_h=(L_{h}^{-\beta}P_h + Q_{h,k,\beta}P_h)(L_{h}^{-\beta}P_h - Q_{h,k,\beta}P_h),
	$
	\begin{align*}
		\|L_{h}^{-2\beta} &P_h - Q_{h,k,2\beta}P_h\|_{\mathcal{L}_2(L_2(\Gamma))} \\
		&\leq (\|L_{h}^{-\beta}P_h\|_{\mathcal{L}(L_2(\Gamma))} + \|Q_{h,k,\beta}P_h\|_{\mathcal{L}(L_2(\Gamma))}) \|L_{h}^{-\beta}P_h - Q_{h,k,\beta} P_h\|_{\mathcal{L}_2(L_2(\Gamma))}\\
		&\leq 2C^{\nicefrac32} e^{-\frac{\pi^2}{2k}} N_h^{\nicefrac12},
	\end{align*}
	which concludes the proof.
\end{proof}

\section{Numerical experiments}\label{sec:experiments}
In this section, we consider the stochastic equation \eqref{Amaineqn}, with $\kappa = 1$,  $\H=1$, and generalized Kirchhoff vertex conditions with $\alpha = \kappa$ on the graph in the left panel of Fig.~\ref{fig:graph}. The aim is to verify the strong error rate of Theorem~\ref{thm:strong_rate} and the error rate for the covariance approximation of Theorem~\ref{cov_fem_approx_rate_quadrature} for the sinc-Galerkin approximations. All experiments were done in \texttt{R} using the \texttt{MetricGraph} 
\cite{bsw_MetricGraph_cran} package.

\begin{figure}[t!]
	\begin{center}
		\includegraphics[width=0.48\linewidth]{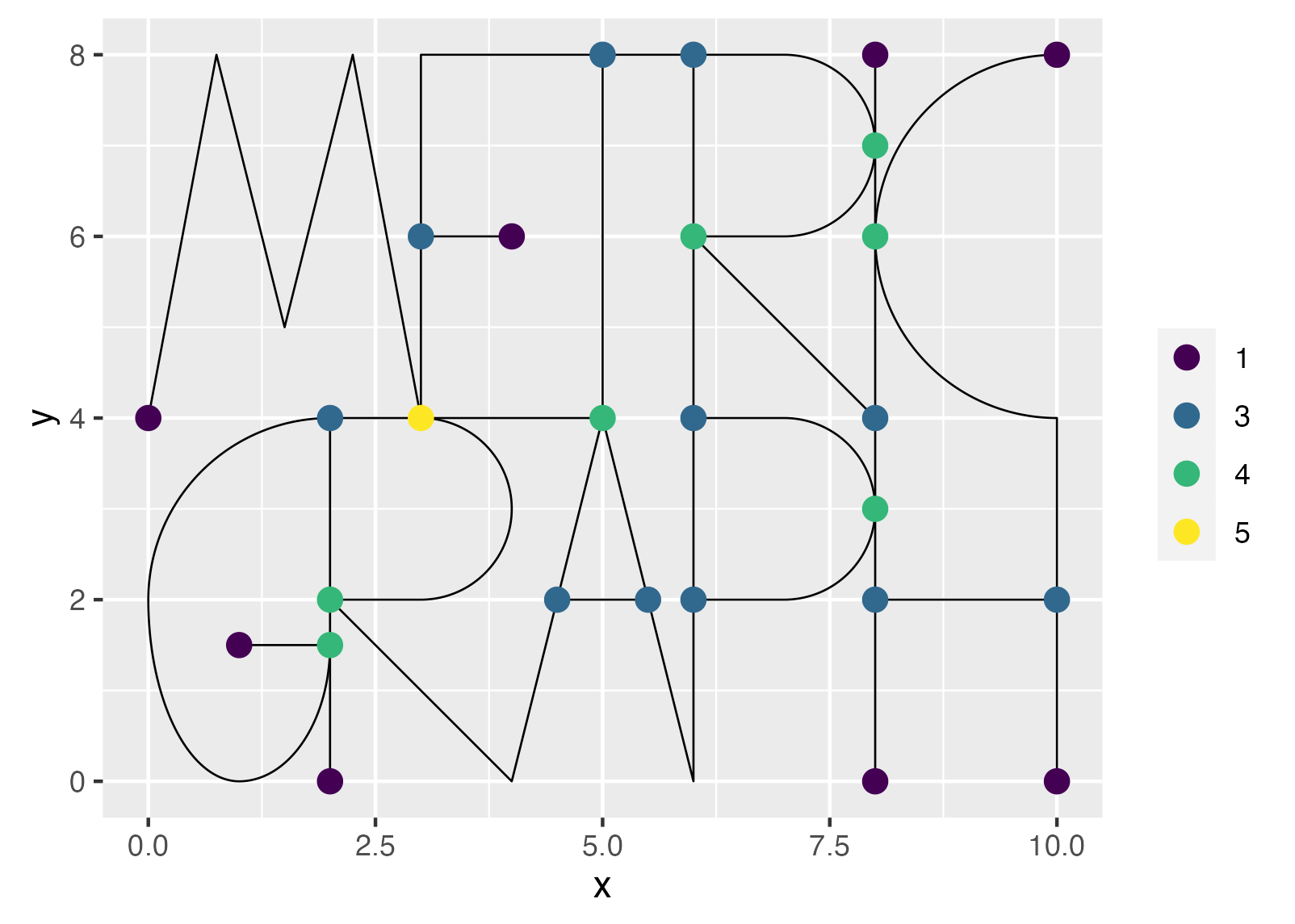}
		\includegraphics[width=0.48\linewidth]{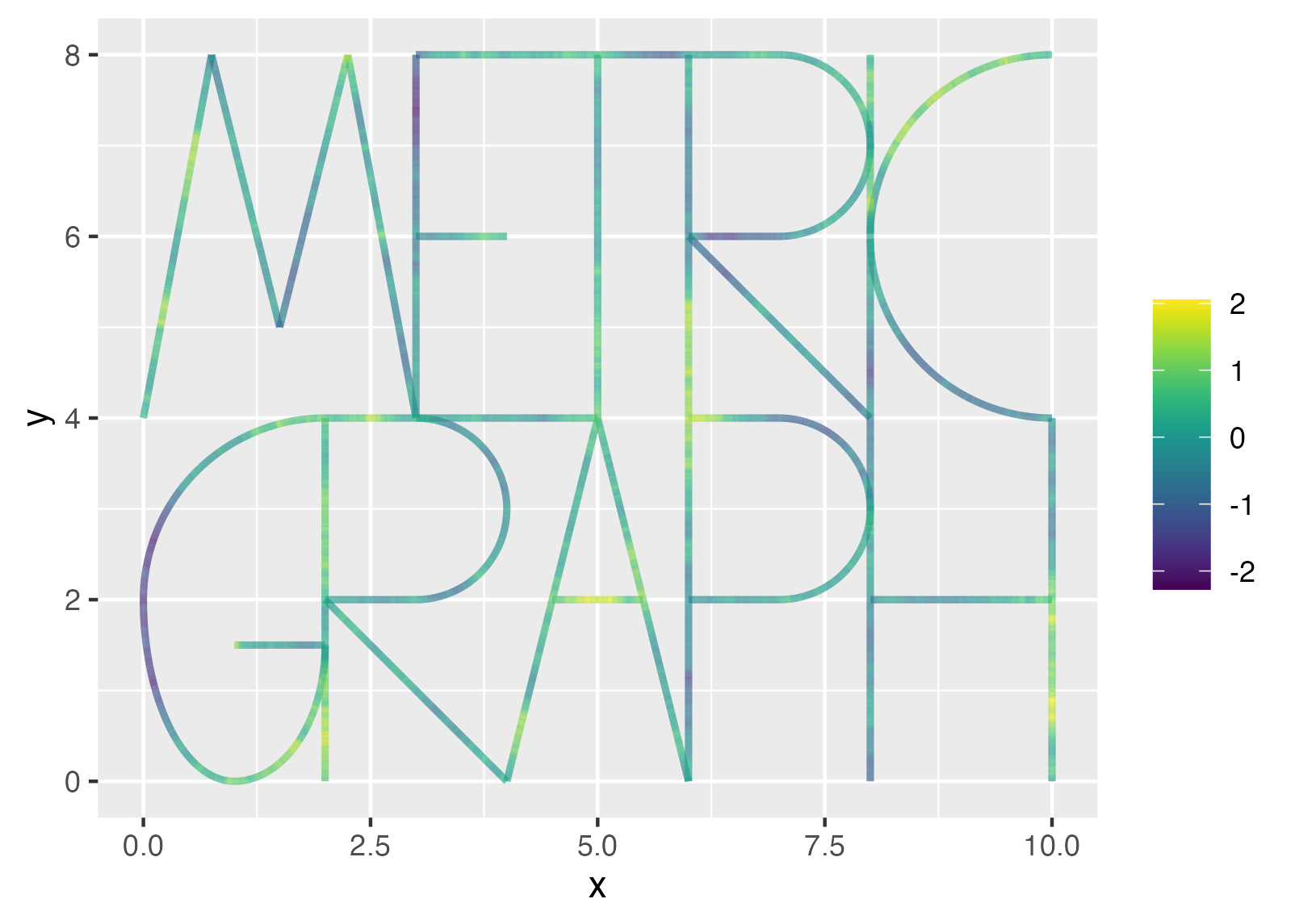}
	\end{center}
	\caption{A metric graph where the colors of the vertices show the degrees (left) and a simulation of the solution to \eqref{Amaineqn} on the graph, with generalized Kirchhoff vertex conditions and $\kappa = \beta = \alpha = 1$.}
	\label{fig:graph}
\end{figure}

\subsection{Strong error}
According to Theorem~\ref{thm:strong_rate}, the theoretical rate of convergence is $2\beta - \nicefrac12$ for $\beta \in (\nicefrac{1}{4},1)$ if the quadrature step size satisfies ${k \leq -\pi^2/(\beta\ln h)}$, where $h$ is the finite element mesh width. We set $k = -1/(\beta\ln h)$ and consider $\beta = \nicefrac{n}{8}$ for $n = 3,4,\ldots,7$ for four meshes, where each edge in the graph is split into equally sized elements on each line, so that the maximal segment length is $2^{-\ell}$, for $\ell = 3,4,5,6$. The resulting number of FEM basis functions and the corresponding numbers of quadrature nodes for each considered value of $\beta$ are listed in Table~\ref{tab:nodenumber}.

\begin{table}[t]
	\caption{The number of finite element basis functions 
		$N_h$ on the considered meshes and
		the corresponding number of quadrature nodes. The last row contains the number for the overkill solution. }
	\begin{center}
		\begin{tabular}{lc|ccccc}
			\toprule
			& $\beta$ & $\nicefrac38$  & $\nicefrac48$ & $\nicefrac58$ & $\nicefrac68$ & $\nicefrac78$\\
			\cmidrule(r){1-7}
			$N_h$ & 799 & 9 & 13 & 20 & 35 & 77 \\
						& 1595 & 14  & 21 & 33 & 59 & 135 \\
						 & 3183 & 20 & 31 & 51 & 91 & 209 \\
						 & 6364 & 28  & 45 & 73 & 131 & 301 \\
			$N_{\mathrm{ok}}$ & 101807 & 73 & 121 & 200 & 357 & 832 \\
			\bottomrule
		\end{tabular}
	\end{center}
	\label{tab:nodenumber}
\end{table}

Because we cannot compute the exact solution, we use a sinc-Galerkin approximation $u_{\mathrm{ok}}$ obtained using an overkill mesh with $h = 2^{-10}$ as a proxy for the true solution while computing the strong mean-square error between the exact solution $u$ and the approximation $u_{h,k}^Q$. Specifically, we first sample the right-hand side $\mv{W}_{\mathrm{ok}} \sim \pN(\mv{0}, \mv{M})$, where $\mv{M}$ is the mass matrix for the FEM basis based on the overkill mesh, see equation \eqref{eq:sample}. Based on this, we compute the $\mv{u}_{\mathrm{ok}}$ overkill solution via \eqref{mainApproximation} evaluated on the overkill mesh. To compute the FEM approximation based on a coarser mesh, we first project $\mv{W}_{\mathrm{ok}}$ onto the corresponding FEM basis by computing $\mv{W}_h = \mv{A}\mv{W}_{\mathrm{ok}}$, where $\mv{A}$ is a matrix with elements $A_{ij} = \varphi_i(s_j)$, where $\varphi_i$ denotes the $i$th basis function based on the coarse mesh and $s_j$ denotes the $j$th node in the overkill mesh. Then, we calculate the corresponding sinc-Galerkin approximation $\mv{u}_h$ using $\mv{W}_h$ as the right-hand side in \eqref{mainApproximation}, and finally, project the solution to the overkill mesh by computing $\mv{u}_{h,k} = \mv{A}\widetilde{\mv{u}}$.

For each value of $\beta$, this procedure is repeated 10 times for various realizations of $\mv{W}_{\mathrm{ok}}$, and the observed strong error, $\mathrm{err}$, is computed as
the square-root of the average squared $L_2$-errors,
	$\mathrm{err}^2 := \tfrac{1}{10} \sum_{i=1}^{10} \bigl(\mv{u}_{\mathrm{ok}}^{(i)} - \mv{u}_{h,k}^{(i)}\bigr)^\top
		\mv{M} \, \bigl(\mv{u}_{\mathrm{ok}}^{(i)} - \mv{u}_{h,k}^{(i)}\bigr)$.
The results are presented in the left panel of Fig.~\ref{fig:errors}.
Based on the observed strong errors, for each value of $\beta$, we compute the observed rate of convergence $\mathrm{r}$ using the least squares estimate of the coefficients in the linear regression ${\ln \mathrm{err} = \mathrm{c} + \mathrm{r} \ln h}$.
The resulting observed rates are in accordance with the theoretical values (Table~\ref{tab:rates}).

\begin{figure}[t!]
	\begin{center}
				\includegraphics[width=0.495\linewidth]{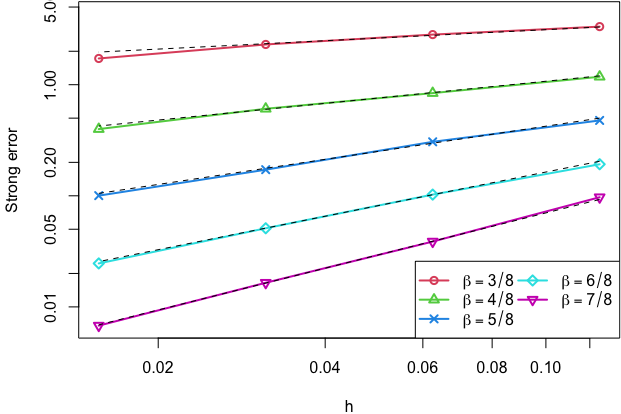}    
				\includegraphics[width=0.495\linewidth]{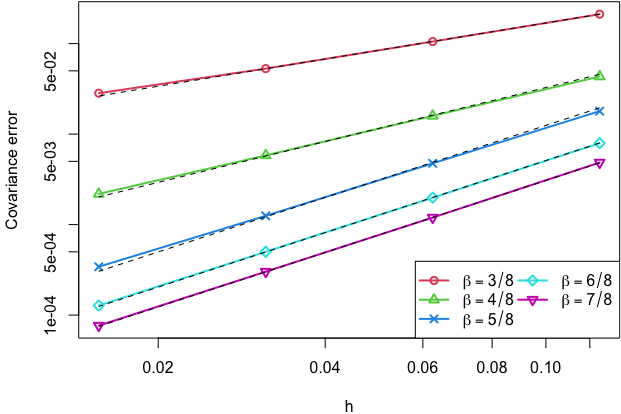}    
	\end{center}
	\caption{Observed strong error (left) and covariance error (right) for different values of $\beta$ as functions of the mesh size $h$. The black dashed lines show the theoretical rates for each case.}
	\label{fig:errors}
\end{figure}

\begin{table}[t]
	\caption{Observed (theoretical) rates of convergence for the strong errors and covariance errors in Fig.~\ref{fig:errors}.}
	{\centering
		\begin{tabular}{lccccc}
			\toprule
			$\beta$ &  $\nicefrac38$  & $\nicefrac48$ & $\nicefrac58$ & $\nicefrac68$ & $\nicefrac78$\\
			\cmidrule(r){2-6}
			Strong rates & 0.31 (0.25) & 0.52 (0.5) & 0.76 (0.75) & 0.99 (1.0)	& 1.27 (1.25)\\
			Covariance rates &  0.97 (1.0)  &  1.44 (1.5) &  1.91 (2.0) & 1.99 (2.0) & 1.99 (2.0)\\
			\bottomrule
	\end{tabular}}
	\label{tab:rates}
\end{table}

\subsection{Covariance errors}
To verify the rate of the covariance approximation given in Theorem~\ref{cov_fem_approx_rate}, we perform a similar experiment to that for the strong errors. 
According to the theorem, the theoretical rate of convergence is ${\min(4\beta - \nicefrac12,2)}$ for $\beta \in (\nicefrac{1}{4},1)$ given the calibration $k = -1/(\beta\ln h)$ for the quadrature step size. We consider the same values of $\beta$ and the same meshes as for the strong errors. 
We again use an overkill approximation, now with $h = 2^{-8}$ to reduce computational costs, as a proxy for the true solution while computing the error between the exact covariance function and approximations. For each case, the $L_2(\Gamma\times\Gamma)$ error is computed by approximating the covariance functions to be piecewise constant on the overkill mesh. 
The results are provided in the right panel of Fig.~\ref{fig:errors}. As for the strong errors, the observed rates of convergence are computed using linear regression and are in 
accordance with the theoretical values (Table~\ref{tab:rates}).

\section{Acknowledgment}
The second author acknowledges the support of the Marsden Fund of the Royal Society of New Zealand (grant no.~18-UOO-143), the Swedish Research Council (VR) (grant no.~2017-04274) and the National Research, Development, and Innovation Fund of Hungary (grant no.~TKP2021-NVA-02 and K-131545).
The third author was supported in part  by a NBHM post-doctoral fellowship from the Department of Atomic
Energy (DAE), Government of India (file no.~0204/6/2022/R\&D-II/5635). 

\bibliographystyle{amsplain}
\bibliography{unified_graph_bib.bib}
\end{document}